\documentclass[11pt]{amsart}

\usepackage{amsmath}
\usepackage{amssymb,amsthm,amsfonts}
\usepackage{amsrefs}
\usepackage[mathscr]{euscript}
\usepackage[margin=1.5cm]{geometry}
\usepackage{graphicx}
\usepackage{wrapfig}
\usepackage[usenames,dvipsnames,svgnames,table]{xcolor}

\renewcommand{\mathcal}{\mathscr}

\def\R {\mathbb{R}}

\def\N {\mathbb{N}}

\def\Z {\mathbb{Z}}

\renewcommand{\epsilon}{\varepsilon}

\newcommand{\e}{\varepsilon}

\renewcommand{\leq}{\leqslant}
\renewcommand{\le}{\leqslant}
\renewcommand{\geq}{\geqslant}
\renewcommand{\ge}{\geqslant}

\DeclareMathOperator*{\osc}{osc}
\newcommand{\Per}{\mathrm{Per}}

\newtheorem{proposition}{Proposition}[section]
\newtheorem{theorem}[proposition]{Theorem}
\newtheorem*{theorem*}{Theorem}

\newtheorem{lemma}[proposition]{Lemma}

\theoremstyle{definition}
\newtheorem{definition}[proposition]{Definition}

\newtheorem{remark}[proposition]{Remark}
\numberwithin{equation}{section}

\begin{document}

\title{Minimizers for nonlocal perimeters of Minkowski type}\thanks{This work has been supported by the Andrew Sisson Fund 2017.\\
{\em Annalisa Cesaroni}:
Dipartimento di Scienze Statistiche,
Universit\`a di Padova, Via Battisti 241/243, 35121 Padova, Italy. {\tt annalisa.cesaroni@unipd.it}
\\
{\em Serena Dipierro}:
Dipartimento di Matematica, Universit\`a di Milano,
Via Saldini 50, 20133 Milan, Italy.
{\tt serena.dipierro@unimi.it}\\
{\em Matteo Novaga}: Dipartimento di Matematica,
Universit\`a di Pisa, 
Largo Pontecorvo 5, 56127 Pisa,
Italy. {\tt matteo.novaga@unipi.it}\\
{\em Enrico Valdinoci}:
School of Mathematics
and Statistics,
University of Melbourne, 813 Swanston St, Parkville VIC 3010, Australia,
Dipartimento di Matematica, Universit\`a di Milano,
Via Saldini 50, 20133 Milan, Italy, and IMATI-CNR, Via Ferrata 1, 27100 Pavia,
Italy. {\tt enrico.valdinoci@unimi.it}}
\author[A. Cesaroni]{Annalisa Cesaroni}
\author[S. Dipierro]{Serena Dipierro}
\author[M. Novaga]{Matteo Novaga}
\author[E. Valdinoci]{Enrico Valdinoci}


\begin{abstract}
We study a nonlocal perimeter functional
inspired by the Minkowski content, whose main feature is that it
interpolates between the classical perimeter and
the volume functional. 
%This problem is  related by a generalized coarea formula
%to a Dirichlet energy
%functional in which the energy density is the local oscillation
%of a function.

This nonlocal functionals  arise in concrete applications,
since
the nonlocal character of the problems and the
different behaviors of the energy  at different scales
allow the preservation of
details and irregularities 
of the image in the process of removing white noises,
thus improving the quality of the image without losing relevant 
features.

In this paper, we provide a series of results
concerning existence, rigidity and classification of minimizers,
compactness results, isoperimetric inequalities,
Poincar\'e-Wirtinger inequalities and density estimates.
Furthermore, we provide the construction of planelike minimizers
for this generalized perimeter under a small and periodic
volume perturbation.
\end{abstract}

\maketitle

\tableofcontents

\section{Introduction}

The main goal of this paper is to study a class of variational problems
which interpolate between the classical perimeter and
the volume functionals. 
Generalized energies of this type have been analyzed  in~\cite{MR3023439, MR3401008, cn, dnv}, 
and also in anisotropic contexts and in view of discretizations methods
in~\cites{MR2655948, MR3187918}.
In terms of applications,
nonlocal functionals
interpolating between perimeter and volume
are often used in image processing
to keep fine details and irregularities 
of the image while denoising additive white noises,
see e.g.~\cites{MR2520779, MR2728706}.

These objects are modeled by
an energy which, at large scales, resembles the perimeter
using an approximation based on the Minkowski content,
but on small scales they present predominant volume contributions,
giving rise to a sort of nonlocal behavior, which may produce severe losses of regularity and compactness properties. 
We will call   {\it $r$-perimeters} these nonlocal perimeter functionals.
%,  {\it $r$-perimeters},  are related to
%a Dirichlet energy functional by a suitable coarea formula.
%The energy density of such Dirichlet functional is also of nonlocal type,
%since it takes into account the oscillation of the function at a small scale.

We recall that in  recent years 
a lot of attention has been devoted to the analysis of other type of nonlocal perimeter functionals, 
starting from the seminal work of Caffarelli, Roquejoffre  and Savin \cite{MR2675483}, 
where it was initiated the study of the Plateau  problem for such kind of nonlocal  perimeters.  
A regularity theory for minimizers of such perimeters has been developed in analogy to the regularity theory of classical minimal surfaces, 
 and also the geometric and variational relation with the classical perimeter has been investigated. For a general overview on the subject we refer to 
\cite{MR3046979} and references therein.

\medskip

In this paper, we develop a preliminary study of the main properties of  the
 $r$-perimeters. First of all we analyze the main features of sets with finite $r$-perimeter, in particular compactness properties, 
to get existence for the Dirichlet problem, and  then isoperimetric inequalities. The  global isoperimetric inequality is a direct consequence of   
 the Brunn-Minkowski inequality, whereas its local version is valid at the appropriate scale. 

We show some rigidity results for minimizers of the $r$-perimeter, in dimension 2, and we 
presents some properties of minimizers.  In particular we consider their density properties, 
pointing out an interesting phenomenon not appearing in the classical case. Indeed in the density estimates  two scales of growth appear:
if the initial density is below a given  threshold depending on $r$, then there is an exponential density growth, then, 
over the threshold,  the growth reduces to the usual one, that is the radius to the power $n$.  

An important feature of these results is that they always need to capture
the ``local'' behavior of the minimizers, which can be rather different than the
``global'' one, due to nonlocal effects at small scales.
In addition, these problems are not scale invariant and they do
not possess any associated extended problem of local type, therefore
many classical techniques related to scaled iterations and monotonicity
formulas are not easily applicable in our setting. 
In particular, we show with a concrete example (see Theorem \ref{FAIL:COM})
that compactness and regularity properties can  fail, at a small scale, for minimizers
of the $r$-perimeter with the addition of a sufficiently large volume term.

Finally the last section is devoted to the construction of plane-like minimizers for the $r$-perimeters in a periodic medium.  
A classical problem in different fields, including geometry,
dynamical systems and partial differential equations,
consists in the determination of objects that are embedded into
a periodic medium and present bounded oscillations with respect to
a reference hyperplane. These objects are somehow the natural extension of ``flat''
objects such as hyperplanes and linear functions and have the
important property that, for these solutions,
the forcing term produced by the lack of homogeneity
of the medium ``averages out'' at a large scale.
We refer to~\cite{MR1501263, MR1503086} for the first results
of this type on geodesics, to~\cite{MR847308} 
for the introduction of this setting in the case of elliptic integrands,
to~\cite{MR1852978, MR2197072, ct}
for the case of hypersurfaces with prescribed mean curvature,
to~\cite{MR2099113, MR845785}
for the case of partial differential equations
and to~\cite{MR2123651, COZZI}
for problems
related to statistical mechanics. 
The planelike structures are also useful to construct pinning effects
and localized bump solutions, see e.g.~\cite{MR2777010, MR2745195}.
See also~\cite{COZZI2}
for planelike constructions related to nonlocal problems of fractional type
and~\cite{2016arXiv160503794C}
for a general review.

We also address the problem of existence of planelike minimizers for energy functionals in which the $r$-perimeter in (1.2) is modulated by a volume term which is periodic and with a sufficiently small size.
This is a setting not comprised in the existing literature, since, as far as we know,
the only nonlocal cases taken into account are the ones arising
from fractional minimal surfaces or related to the Ising model.

\medskip

In the rest of this section, we formalize the mathematical setting in which
we work and we present our main results.
The nonlocal perimeter
functional based on Minkowski content 
%and the corresponding Dirichlet energy
%based on the local oscillation of a function
will be introduced in
Subsection~\ref{A nonlocal perimeter functional}.
Some rigidity properties of minimizers will be also discussed.
%Via a coarea formula, it is possible to relate this nonlocal perimeter to 
%a nonlocal Dirichlet energy based on the local oscillation of a function. 

In Subsections~\ref{Compactness properties}
we present some 
compactness
results at large scales and 
some $\Gamma$-convergence results for this nonlocal perimeter.
%and for the corresponding nonlocal Dirichlet energy, respectively. 
%
Then, in Subsection~\ref{diri:sec} we discuss the Dirichlet problem
%related to these nonlocal functionals,
and in Subsection~\ref{sec:class}
we present some rigidity results.

In Subsection~\ref{Isoperimetric SE} 
we introduce global and relative
isoperimetric inequalities. Furthermore, we provide
density estimates for the nonlocal perimeter,
which in turn show that the compactness 
and regularity properties of the nonlocal perimeter
minimizers may be deeply influenced by oscillations at small scales.

Finally, the planelike minimizers for the nonlocal 
perimeter are discussed in
Subsection~\ref{Existence of planelike}.

A detailed organization of the paper is then presented
at the end of the Introduction, in Subsection~\ref{sec:org}.

\subsection{The nonlocal perimeter and the corresponding Dirichlet energy}
\label{A nonlocal perimeter functional}

We start with some preliminary definitions.
Given~$r>0$ and~$E\subseteq\R^n$, we let
\begin{equation}\label{omega r}\begin{split} & E\oplus B_r:=
\bigcup_{x\in E} B_r(x)=(\partial E\oplus B_r)\cup E
=(\partial E\oplus B_r)\cup(E\ominus B_r),\\
{\mbox{where}}\qquad  &
E\ominus B_r:=
E\setminus \left(\bigcup_{x\in\partial E} B_r(x)\right)=E\setminus\big(
(\partial E)\oplus B_r\big).\end{split}\end{equation}
We shall identify a set~$E\subseteq \R^n$ with its points of density one and 
$\partial E$ with the topological boundary of the set of points of density one. 
 %to be the boundary  in the measure theoretic
%sense, namely we say that~$p\in\partial E$ if
%for any~$\rho>0$ we have that~${\mathcal{L}}^n(B_\rho(p)\cap E)>0$
%and~${\mathcal{L}}^n(B_\rho(p)\setminus E)>0$.

Given~$r>0$ and a domain $\Omega\subseteq\R^n$, for any measurable set~$E\subseteq\R^n$, we use the notation in~\eqref{omega r}
and we consider the functional
\begin{equation} \label{per r}
\Per_r(E,\Omega):=\frac{1}{2r}{\mathcal{L}}^n \Big(
\big((\partial E)\oplus B_r\big) 
\cap\Omega\Big)=\frac{1}{2r}
{\mathcal{L}}^n \big(( \partial E)_r\cap\Omega\big).\end{equation}
As customary, we denoted here by~${\mathcal{L}}^n$ the $n$-dimensional
Lebesgue measure. 
When~$\Omega=\R^n$, we simply write~$\Per_r(E):=\Per_r(E,\R^n)$.
Note that our definition agrees with that in \cite{MR2655948, MR3187918}, since we are identifying a set with its points of density one,
therefore 
$$\Per_r(E,\Omega)= \min_{|E'\Delta E|=0} \Per_r(E', \Omega).$$ 
We observe (see e.g. ~\cite{MR3023439}) that $\Per_r$ is weak  lower semicontinuous in $L^1_{\rm{loc}}$ 
and that, for every $A, B$ measurable sets, 
$$
\Per_r(A\cap B, \Omega)+\Per_r(A\cup B, \Omega)\leq 
\Per_r(A, \Omega)+\Per_r(B, \Omega).$$

The definition of~$\Per_r$ is inspired by
the classical Minkowski content (which would be recovered in
the limit, see e.g.~\cite{federer, MR2655948, MR3187918}).
In particular, for 
sets with compact and $(n-1)$-rectifiable boundaries, the functional in~\eqref{per r}
may be seen as a nonlocal approximation of the classical perimeter functional,
in the sense that
\begin{equation*} \lim_{r\searrow0}  \Per_r(E)=
{\mathcal{H}}^{n-1} (\partial E).\end{equation*}
Hence, in 
some sense, $\Per_r$ recovers a perimeter functional for small~$r$
and a volume energy for large~$r$.

We observe that recently a great attention has been devoted
to the fractional perimeters
introduced in~\cite{MR2675483}, which also interpolate
the classical perimeter with an area type functional (see e.g.~\cite{GP}
for a review on such topic).
The functional in~\eqref{per r} is however
very different in spirit from that in~\cite{MR2675483},
since the lack of scaling invariance does not
allow a classical regularity theory 
and causes severe lack of compactness
at small scales (as we will discuss in details in the sequel).

More generally,  given~ a domain~$\Omega\subseteq\R^n$ and a function  $g\in  L^1_{\rm loc}(\R^n)$, 
we define the energy 
\begin{equation}\label{per rg} {\mathcal{F}}_{r,g} (E,\Omega):=\Per_r(E,\Omega)+\int_{E\cap\Omega} g(x)\,dx.\end{equation} 

%%%%%%%%%%%%%%%%%%
%%%%%%%%%%%%%%%%%%%%%%%%%%%%%%%
%%%%%%%%%%%%%%        DEFINIZIONE OSCILLAZIONE

The functional in~\eqref{per r} is related via a coarea formula
to a Dirichlet energy which takes into account the
local oscillation of a function, which is described as follows. 
Let $\Omega\subseteq \R^n$ be a  domain  and $u\in L^1_{\rm loc}(\Omega\oplus B_r)$. Then for any~$x\in\Omega$ we 
consider the oscillation of~$u$ in~$B_r(x)$,
given by
$$ \osc_{B_r(x)} u:= \sup_{B_r(x)} u-\inf_{B_r(x)} u.$$
In this paper, in the $\sup$ and $\inf$ notation, we mean
the ``essential supremum and infimum'' of the function (i.e.,
sets of null measure are neglected).
It can be checked by using the definition that a triangular inequality holds: for all $v,u\in  L^1_{\rm loc}(\Omega\oplus B_r)$  
$$ \osc_{B_r(x)} (u+v)\leq \osc_{B_r(x)} (u)+\osc_{B_r(x)} (v).$$
We have the  following generalized coarea formula
(see formulas~(4.3) and~(5.7)
in~\cite{MR3401008} for similar formulas in very related contexts):

\begin{lemma}\label{COAREA:L}
It holds that
\begin{equation}\label{COAREA}
\int_\Omega \osc_{B_r(x)} u\,dx=2r\,\int_{-\infty}^{+\infty}
\Per_{r}(\{u>s\}, \Omega) \,ds.\end{equation}
\end{lemma}

%Given~$p\ge1$
%and an open set~$\Omega\subseteq\R^n$, we now introduce the following functional:
%\begin{equation}\label{E r} {\mathcal{E}}_{r, p}(u, \Omega):=
%\int_\Omega \left( \osc_{B_r(x)} u \right)^p\,dx.\end{equation}
%% We introduce also in this case the notion of  local minimizers 
%% and Class~A minimizers. 
%This functional is $p$-homogeneous, convex (then subadditive) and  
%weak  lower semicontinous in $L^1_{\rm{loc}}$  see e.g. ~\cite{MR3023439}. 
%This implies  that also $\Per_r$ is weak  lower semicontinous in $L^1_{\rm{loc}}$ 
%and that, for every $A, B$ measurable sets, 
%$$
%\Per_r(A\cap B, \Omega)+\Per_r(A\cup B, \Omega)\leq 
%\Per_r(A, \Omega)+\Per_r(B, \Omega).$$

%%%%%%%%%%%%%%%%%%%%%%%%%%%%%%%%%%%
%%%%%%%%%%%%%%%%%%%%%%%%%%%%      DEFINIZIONE MINIMI CLASSE A 

In the setting of~\eqref{per r} and~\eqref{per rg},  we introduce the  definition of local minimizer and Class~A minimizer. 
We are interested in existence, compactness and regularity properties of such minimizers. Moreover we will also provide construction of planelike minimizers for such energies in periodic media. 

\begin{definition}[Local minimizer and Class~A minimizer]
\label{DEFIN4}
A set~$E$ is a  minimizer for  ${\Per_r}$ (resp. for ${\mathcal{F}}_{r,g}$) in a bounded domain~$\Omega$
if for any measurable set~$F\subseteq\R^n$
with~$F\setminus(\Omega\ominus B_r)=E\setminus(\Omega\ominus B_r)$ it holds that
$$\Per_r(E,\Omega)\le\Per_r(F,\Omega)\qquad (\text{resp. }  {\mathcal{F}}_{r,g}(E, \Omega)\leq {\mathcal{F}}_{r,g}(F, \Omega)).$$
$E$ is
a  Class~A minimizer if it is a minimizer in this sense in any ball of~$\R^n$.
\end{definition} 

We observe that if~$E$ is a Class~A minimizer, then
$$ \Per_r(E, B_R)\le n\omega_n R^{n-1}, \quad {\mbox{ if }} R\ge 2r.$$
Indeed, 
$$ \Per_r(E, B_R)\le \Per_r(E\setminus B_{R-r}, B_R)\le
\Per_r(B_{R-r}, B_R)= \frac{\omega_n}{2r}\left( R^n-(R-2r)^n\right).$$

Note that in Definition~\ref{DEFIN4}
we allow competitors only away from the boundary of the domain, in a way compatible with the natural
scale of the problem. Actually, 
this is the appropriate notion of minimizer, 
since the following result shows that the problem trivializes if competitors
are allowed to produce modifications up to the boundary of the domain. 
%
%\begin{definition}[Local strong minimizer and strong Class~A minimizer] 
%\label{DEFIN3}
%A set~$E$ is a strong  minimizer for  ${\Per_r}$ (resp. for ${\mathcal{F}}_{r,g}$) in a domain~$\Omega$
%if for any measurable set~$F\subseteq\R^n$
%with~$F\setminus\Omega=E\setminus\Omega$ it holds that
%$$ \Per_r(E,\Omega)\le\Per_r(F,\Omega)\qquad (\text{resp. }  {\mathcal{F}}_{r,g}(E, \Omega)\leq {\mathcal{F}}_{r,g}(F, \Omega)).$$
%$E$ is
%a strong Class~A minimizer if it is a minimizer in any ball of~$\R^n$.
%\end{definition} 

\begin{proposition}\label{TRIVIA}
Let~$E\subseteq\R^n$ be such that for every ball $B\subseteq \R^n$, and for any measurable set~$F\subseteq\R^n$
with~$F\setminus B=E\setminus B$ it holds that
\[ \Per_r(E,B)\le\Per_r(F,B).\]
 Then either~$E=\varnothing$ or~$E=\R^n$.
\end{proposition}

\subsection{$\Gamma$-convergence results and compactness properties}\label{Compactness properties}
We start with some convergence results on $\Per_{r_k}$  as $r_k\to r$. We focus also on compactness properties
of sets with bounded energy.
 
\begin{theorem}\label{SICOM:01} 
Let~$\Omega$ be either an open and bounded subset of~$\R^n$,
or equal to~$\R^n$.
Let also $r_k\to r\in (0,+\infty)$.
Then the following holds.
\begin{enumerate} 
\item $\Per_{r_k}(E, \Omega) \to \Per_{r}(E, \Omega)$ and $\mathcal{F}_{r_k,g}(E, \Omega) \to \mathcal{F}_{r,g}(E, \Omega)$ for all $E\subseteq \R^n$. 
\item $\Per_{r_k}(\cdot, \Omega)$ (resp. $\mathcal{F}_{r_k,g}(\cdot, \Omega)$) $\Gamma$-converges in $L^1_{\rm{loc}}(\Omega)$ to 
$\Per_{r}(\cdot, \Omega)$ (resp. to $\mathcal{F}_{r,g}(\cdot, \Omega)$).
\item   Let $E_k\subseteq\R^n$ be such that 
$$\sup_{k\in\N} \Per_{r_k}(E_k,\Omega)<+\infty.$$
Assume that, up to subsequences, 
$\chi_{E_k}\rightharpoonup u$, in $L^1_{\rm{loc}}\big(\Omega)$
with $u:\R^n\to [0,1]$.  Let $\Sigma:= \{x\in\Omega: u(x)\in (0,1)\}$ and 
\begin{equation} \label{ELL2} \ell:=\liminf_{r_k\to r} \Per_{r_k}(E_k,\Omega).\end{equation}\\
Then the following holds true:
\begin{equation}\label{ZZ-1}
{\mathcal{L}}^n(\Sigma\oplus B_r\cap \Omega)\le 2\ell r ,
%\qquad\qquad{\mathcal{H}}^{n-1}(\partial \Sigma)\le C\ell
\qquad \chi_{E_k} \to u \;{\mbox{ in }}\;L^1_{\rm{loc}}(\Omega
\setminus \Sigma\big), \qquad
 \int_\Omega \osc_{B_r(x)}u\, dx\leq 2r \ell.
\end{equation}
\end{enumerate}  
\end{theorem} 

Observe that  we cannot expect a stronger compactness result, due to the following observation. 
\begin{remark}\label{NOCOM}
Families of sets~$E_k$ for which~$\Per_r(E_k,\Omega)\le 1$
are not necessarily compact in~$L^1(\Omega)$ (and, more generally,
it is not necessarily true that~$\chi_{E_k}$ converges pointwise
up to a subsequence).
\end{remark}

\begin{remark} \label{REML01293dfas}
When~$\Omega$ is unbounded and~$r_k\searrow r>0$,
some pathological counterexamples to the claim in~(1)
of Theorem~\ref{SICOM:01} may arise. For instance, one may have that
\begin{equation}\label{0q9weir9wurgrodighsdifhw}
\Per_{r_k}(E,\Omega)=+\infty 
\quad{\mbox{ while }}\quad
\Per_{r}(E,\Omega)=0.\end{equation}
\end{remark}

%\begin{theorem}\label{SICOM:01}
%Let~$\Omega\subseteq\R^n$ be open and bounded.
%We consider sequences $\Omega_k\supseteq\Omega$, $E_k\subseteq\R^n$ and 
% $r_k\to r\in (0,+\infty)$ as~$k\to+\infty$ such that 
%\begin{equation} \label{ELL1}
%\sup_{k\in\N} \Per_{r_k}(E_k,\Omega_k)<+\infty.\end{equation}
%
%Then, there exist~$E\subseteq \R^n$, $Z\subseteq\R^n$ 
%and a constant~$C>0$ only depending on~$n$ such
%that
%\begin{eqnarray}
%\label{0001}
%&& \partial E\subseteq Z,\\
%&&\label{ZZ-1}
%{\mathcal{L}}^n(Z\oplus B_r)\le C\ell r ,\qquad\qquad
%{\mathcal{H}}^{n-1}(\partial Z)\le C\ell
%,\\ 
%\label{ZZ-2}
%{\mbox{and }}&& \chi_{E_k} \to \chi_E \;{\mbox{ in }}\;L^1\big((\Omega\ominus B_r)
%\setminus Z\big), \;{\mbox{ up to a subsequence}}
%\end{eqnarray}
%where 
%\begin{equation} \label{ELL2} \ell:=\liminf_{k\to+\infty} \Per_{r_k}(E_k,\Omega).\end{equation}
%In addition, one has that
%\begin{equation}\label{SELEZ}
%\liminf_{k\to+\infty} \Per_r(E_k,\Omega)\ge\Per_r(E,\Omega).
%\end{equation}
%\end{theorem}
%
%As a matter of fact, a compactness version in the spirit of Theorem~\ref{SICOM:01}
%also holds in case~$r_k\to0$, 
%with a slightly more involved, but useful, statement:

We recall also the following result, which is proved in \cite[Theorem 3.1 and Remark 3.4]{MR3187918}.

\begin{theorem}\label{SICOM} 
Let~$\Omega\subseteq\R^n$ be open and bounded, and~$r\to 0$.
Then the following holds.
\begin{enumerate} 
\item $\Per_{r}(\cdot, \Omega)$ $\Gamma$-converges in $L^1_{\rm{loc}}(\Omega)$ to 
$\Per(\cdot, \Omega)$, where $\Per$ is the standard perimeter. 
\item Let  $E_r\subseteq\R^n$ be such that 
\begin{equation*}
\sup_{r} \Per_{r}(E_r,\Omega)<+\infty.\end{equation*}
Then there exists $E\subseteq \R^n$ such that 
% and a constant~$C>0$ only depending on~$n$ such that 
\[ \chi_{E_r}\to \chi_E \quad{\mbox{ in }} L^1_{\rm{loc}}(\Omega)\quad{\mbox{up to a subsequence}}\]
and \[\Per(E,\Omega)\leq \liminf_{r\to 0} \Per_r(E_r,\Omega).\]
\end{enumerate} 

An analogous result holds for the functional~$\mathcal{F}_{r,g}$. 
\end{theorem} 
 
%Let~$\Omega\subseteq\R^n$ be open and bounded.
%Consider a sequence~$\Omega_k
%\supseteq\Omega$, with 
%\begin{equation}\label{LIPOM}
%{\mbox{$\partial \Omega_k$ is locally Lipschitz, with Lipschitz constant bounded
%uniformly in $k$.}} 
%\end{equation}
%Let 
%\begin{equation}\label{UNIKArk} 
%r_k\to 0\;{\mbox{ as }}\;k\to+\infty\end{equation}
%and $E_k\subseteq\R^n$ such that 
%\begin{equation}\label{UNIKA}
%\sup_{k\in\N} \Per_{r_k}(E_k,\Omega_k)<+\infty.\end{equation}
%Then there exists $E\subseteq \R^n$ such that and a constant~$C>0$ only depending on~$n$ such
%that
%\[ \chi_{E_k}\to \chi_E \quad{\mbox{ in }} L^1(\Omega)\quad{\mbox{
%up to a subsequence}}\]
%and \[\liminf_{k\to+\infty} \Per_r(E_k,\Omega_k)\ge C\Per(E,\Omega).\]
%
% 
%Then, there exist~$\widehat E_k\subseteq \R^n$, $E\subseteq\R^n$ 
%and a constant~$C>0$ only depending on~$n$ such
%that
%\begin{eqnarray}
%\label{0001:INCLU}
%&&\widehat E_k\supseteq E_k,\\
%\label{0001:BIS}
%&& \Per( \widehat E_k,\,
%\Omega_k )\le C\,\Per_{r_k}(E_k,\Omega_k)\\
%\label{0002:BIS}
%&& \int_{\Omega_k }
%|\chi_{E_k}-\chi_{\widehat E_k}|\,dx\le
%C\,r_k\,\Per_{r_k}(E_k,\Omega_k)
%\\ \label{0003:BIS}
%{\mbox{and }}&& \chi_{E_k}\to \chi_E \quad{\mbox{ in }} L^1(\Omega)\quad{\mbox{
%up to a subsequence.}}
%\end{eqnarray}
%\end{theorem}
%
%The general statement in Theorem~\ref{SICOM} will turn out to be
%useful in the proof of the forthcoming relative isoperimetric
%inequality in Theorem~\ref{ISOPER}.\medskip

%%%%%%%%%%%%%%%%%%%%%%%%%%%%%%%%%%%%%%%%%%%%%%%%%%%%%%%%%%%%%%%      COMPATTEZZA MINIMI 
Dealing with compactness issues, it
is interesting to point out that sequences of minimizers
(differently than sequences of sets with bounded $r$-perimeter)
always provide a limit which is also a minimizer, on a smaller set.
{F}rom the technical point of view, such limit is obtained
from the support of the weak limit, once sets of zero measure are neglected. 
The precise statement goes as follows:

\begin{theorem}\label{promin}
Let $\Omega$ be a bounded open set and 
$E_k$ be a sequence of local minimizers of ${\mathcal{F}}_{r,g}$ in $\Omega$ such that  $\chi_{E_k}\rightharpoonup u$, 
with $u:\R^n\to [0,1]$, in $L^1(\Omega)$. 

Let $E$ be such that $\{ u=1\}\subseteq E\subseteq \Omega\setminus \{u=0\}$, 
and $\Sigma:= \{x\in\Omega : u(x)\in (0,1)\}$.
Then $E$ is a  local minimizer of ${\mathcal{F}}_{r,g}$ in $\Omega\ominus B_r$ and $g(x)=0$ for a.e. $x\in \Sigma$.

Moreover, if  $\mathcal L^n(\{ g=0\})=0$ then $\chi_{E_k}\to \chi_E$.  
\end{theorem}

\subsection{The Dirichlet problem}\label{diri:sec}
 
We now  consider the Dirichlet problem 
for the functional $\Per_r$.

\begin{theorem}\label{DEBP}
Let~$E_o\subseteq\R^n$ and $\Omega$ a bounded open set. Fix $\Omega'\Subset \Omega$.
Then, there exists~$E\subseteq\R^n$ such that~$E\setminus\Omega'=
E_o\setminus \Omega'$, and
$$ \Per_r(E,\Omega)\le\Per_r(F,\Omega)$$
for any~$F\subseteq\R^n$ for which~$F\setminus\Omega'
=E_o\setminus \Omega'$.

The same holds for the functional~$\mathcal{F}_{r,g}$. 
\end{theorem}

\subsection{Class~A minimizers} \label{sec:class}
%%%%%%%%%%%%%%%%%%%%%%%%%%%%%%%%%%%%%%%%%%%%%%%%%%%%%%%%%%%%%%%        MINIMI CLASSE A  

In this subsection, we present some rigidity results for the nonlocal functionals
introduced in~\eqref{per r}.  

Next result shows that half-spaces are always 
Class~A minimizers for $\Per_r$.

\begin{proposition}\label{piani}
Let  $\omega\in\R^n$ and $E=\{x\ | x\cdot \omega<0\}$. 
Then $E$ is a Class~A minimizer for~$\Per_r$. 
\end{proposition}

In addition, we give  the complete
characterization of Class~A minimizers in dimension~$1$,
according to the following result:

\begin{theorem}\label{DI1PR}
If $E$ is a Class~A minimizer for ${\Per _r}$ and $n=1$, then $E$ is
either $\varnothing$ or $\R$ or a
halfline of the type either~$(a, +\infty)$ or~$(-\infty, a)$,
for some~$a\in\R$. 
\end{theorem}

%In dimension $n>1$, we provide an asymptotic characterization of Class~A minimizers. 
%\begin{theorem}\label{dimn}
%Let $E$ be a Class~A minimizer for ${\Per _r}$ and let $E_\eps:=\eps E$ for $\eps>0$. 
%Then up to a subsequence $\lim_{\eps\to 0} E_\eps= \widetilde E$  in $L^1_{\rm{loc}}(\R^n)$ where $\widetilde E$ is a 
%Class~A minimizer for the standard perimeter ${\Per}$. 
%
%As a consequence if $n\leq 7$, $\widetilde E$ is either $\varnothing$ or $\R^n$ or a halfspace. 
%
%\end{theorem}
It would be interesting to study  the  Bernstein problem for $\Per_r$. 
In particular, in analogy with the classical perimeter, one could expect that the 
 the Class~A minimizers are only  $\varnothing$ or $\R^n$ 
 or half-spaces, at least in small dimension. 
 
%For the functional $\mathcal{E}_{r,1}$, 
%we have the following monotonicity
%result in dimension~$1$,
%which is based on the classification of Class~A minimizers 
%provided in Theorem \ref{DI1PR}.
%
%\begin{theorem}\label{LAR}
%A function~$u\in L^1_{\rm loc}(\R)$ is
%a Class~A minimizer for 
%the functional in~\eqref{E r} with~$p:=1$ if and only if it is monotone.
%\end{theorem}

%%%%%%%%%%%%%%%%%%%%%%%%%%%%%%%%%%%%%%%%%%%%%%%%%%%%%%%%%%%%%%%        ISOPERIMETRICA 

\subsection{Isoperimetric inequalities and density estimates}\label{Isoperimetric SE}

We now discuss the isoperimetric properties of the functional~$\Per_r$. To this end, we first
point out that balls
are isoperimetric for the functional in~\eqref{per r},
as a consequence of the Brunn-Minkowski Inequality.
Namely, we have that:

\begin{lemma}\label{ISOR}
\begin{itemize}
\item[(i)]
For any~$R>0$ and any measurable set~$E\subseteq\R^n$ 
such that~${\mathcal{L}}^n(E)=
{\mathcal{L}}^n(B_R)$ it holds that
\begin{equation}\label{FORMULA ISP} \Per_r(E)\ge\Per_r(B_R).\end{equation}
\item[(ii)] Viceversa,
if~${\mathcal{L}}^n(E)=
{\mathcal{L}}^n(B_R)$ and
$$ \Per_r(E)=\Per_r(B_R),$$
then~$E=B_{ R}(p)\setminus{\mathcal{N}}$,
for some set~${\mathcal{N}}$ of null measure and some~$p\in\R^n$.
\end{itemize}\end{lemma}

We present now a version of
the relative isoperimetric inequality for~$\Per_r$
in an appropriate scale:

\begin{theorem}\label{ISOPER}
Let assume there exists $\lambda\geq 1$ such that 
\begin{equation} \label{MAGGIORE}
\lambda R\ge r>0.\end{equation} There exists~$C>0$, possibly depending on~$n$, such that
for all ~$E\subseteq\R^n$ with 
\begin{equation}\label{178:89ooo}
\frac{ {\mathcal{L}}^n (E\cap B_R)}{
{\mathcal{L}}^n (B_R) }\le \frac12,
\end{equation}
 there holds
\begin{equation}\label{78:89ooo}\Big(
{\mathcal{L}}^n (E\cap B_R)\Big)^{\frac{n-1}{n}} \le C\lambda\,\Per_r(E,B_R).
\end{equation}
\end{theorem}

For the proof of this result we will need the 
following technical lemma (which can be seen as a
working version of the compactness result in
Theorem~\ref{SICOM}).

\begin{lemma} \label{lemmatecnico} Let~$\Omega\subseteq\R^n$ be open and bounded.
Consider a sequence of sets  $\Omega_k
\supseteq\Omega$, such that 
\begin{equation}\label{LIPOM}
{\mbox{$\partial \Omega_k$ is  uniformly locally Lipschitz.
%, with Lipschitz constant bounded uniformly in $k$.
}} 
\end{equation}
Let 
\begin{equation}\label{UNIKArk} 
r_k\to 0\;{\mbox{ as }}\;k\to+\infty\end{equation}
and $E_k\subseteq\R^n$ such that 
\begin{equation}\label{UNIKA}
\sup_{k\in\N} \Per_{r_k}(E_k,\Omega_k)<+\infty.\end{equation}
 
Then, there exist~$\widehat E_k\subseteq \R^n$,   $E\subseteq \R^N$ 
and a constant~$C>0$ only depending on~$n$ such
that
\begin{eqnarray}
\label{0001:INCLU}
&&\widehat E_k\supseteq E_k,\\
\label{0001:BIS}
&& \Per( \widehat E_k,\,
\Omega_k )\le C\,\Per_{r_k}(E_k,\Omega_k)\\
\label{0002:BIS} 
&& \int_{\Omega_k }
|\chi_{E_k}-\chi_{\widehat E_k}|\,dx\le
C\,r_k\,\Per_{r_k}(E_k,\Omega_k)\\ 
\label{nuovaconv}{\mbox{and }} && 
\chi_{E_k}\to\chi_E \qquad {\mbox{in $L^1(\Omega)$ up to a subsequence.}}
\end{eqnarray}
\end{lemma}

\begin{remark}\label{FAIL}
We stress that ~\eqref{78:89ooo} holds with a constant which depends on $\lambda$, that is, on the ratio $\frac{r}{R}$, if $r>R$. 
Namely, if~$r>R$, \eqref{78:89ooo} may fail to be true for $C$ just depending on $n$.
\end{remark}

As a simple consequence of Theorem~\ref{ISOPER},
we also provide
the following nonlocal
Poincar\'e-Wirtinger inequality:

\begin{theorem}\label{PW}
There exists~$C>0$, only depending on~$n$, such that the following statement holds.
Let $\lambda\geq 1$, with $\lambda R\ge r>0$, and~$u\in L^\infty(B_r)\cap L^1(B_R)$.
Let
$$ \langle u \rangle_R := \frac{1}{ {\mathcal{L}}^n(B_R) }\int_{B_R} u.$$
Then,
\begin{equation} \label{89w1:NOATo}
\int_{B_R} \big|u - \langle u \rangle_R\big|\le \frac{CR\lambda}{ r}
\int_{B_R} \osc_{B_r(x)} u\,dx.\end{equation}
\end{theorem}

\begin{remark}\label{NONPO}
When~$r>R$, the estimate~\eqref{89w1:NOATo} does not necessarily hold true 
with a constant independent of $\lambda$.
\end{remark}

%%%%%%%%%%%%%%%%%%
%%%%%%%%%%%%%%%%%%%
%%%%%%%%%%%%%%%%%%%%%%%%%%      DENSITA'

We address now the density properties of the minimizers of~$\Per_r$.
Differently than the classical cases, the density properties
of the minimizers
may depend on the initial density for small scales: nevertheless,
we can obtain a density growth in larger balls, 
and the constants become uniform once a suitable density threshold
is reached. More precisely, our result is the following:

\begin{theorem}\label{DENSITY}
Let~$\Omega\subseteq\R^n$, $r>0$ and~$E$ be a minimizer for~$\Per_r$.
Let~$R_o>0$. Suppose that $B_{R_o}\subseteq\Omega$
and
\begin{equation}\label{GA0}
\omega_o:={\mathcal{L}}^n (E\cap B_{R_o})>0.\end{equation}
Let also~$k\in\N$ be such that~$B_{R_o+2kr}
\subseteq\Omega$. Then, 
\begin{equation}\label{DENS:EQ1}
{\mathcal{L}}^n(E\cap B_{R_o+2kr})\ge (\omega_o^{\frac{1}{n}}+ 2 c_{\star } k r)^{n},\end{equation}
for a suitable~$c_{\star }>0$, possibly depending on $n$,
$r$ and  $\omega_o$.

Moreover,
\begin{equation}\label{DENS:EQ2}
{\mbox{ if $n=1$, $c_{\star }$ is a pure number, independent
of~$r$ and~$\omega_o$.}}
\end{equation}
Also,
\begin{equation}\label{DENS:EQ3}
{\mbox{if $\omega_o\ge \underline{c}\, r^n$
for some~$\underline{c}>0$, then $c_{\star }$ only
depends on $n$ and $\underline{c}$,
and it is independent of~$\Omega$ and~$\omega_o$.}}
\end{equation}

In addition, if
\begin{equation}\label{DENS:EQ5}
{\mathcal{L}}^n(E\cap B_{R_o+2(k-1)r})\le \overline{C} r^n,\end{equation}
for some~$\overline{C}>0$, then
\begin{equation}\label{DENS:EQ4}
{\mathcal{L}}^n(E\cap B_{R_o+2(k-1)r})\ge 
\omega_o
\,\left(1+\widetilde c\right)^k,
\end{equation}
for some~$\widetilde c>0$, depending on~$n$ and~$\overline{C}$.
\end{theorem}

It is interesting to point out that
Theorem~\ref{DENSITY} detects two scales of growth
(and this fact is different from the case of the classical
minimal surfaces, as well as of the nonlocal minimal surfaces
in~\cite{MR2675483}, where there is only
one type of growth, given by the dimension of the space).
Indeed, in our framework,
if the initial density is below the threshold prescribed
by~$r^n$ (as stated in~\eqref{DENS:EQ5}),
then there is an exponential density growth
(as stated in~\eqref{DENS:EQ4}), till the density reaches
the quantity~$r^n$.
Then, once a density of order~$r^n$ is reached,
the growth reduces to the usual one, that is the radius
to the power~$n$ (as stated in~\eqref{DENS:EQ1}).
In such case of polynomial growth
away from an initial~$r^n$, the constant
become uniform (as stated in~\eqref{DENS:EQ3},
being the onedimensional case special, in view
of~\eqref{DENS:EQ2}).

We think that it would be interesting to establish whether or not
the growth in~\eqref{DENS:EQ4} is optimal or if sharper estimates
may be obtained independently on the initial density.

%%%%%%%%%%%%%%%%%%%%%%%%%%%%%%%%%
%%%%%%%%%%%%%%%%%%%%%%%%%%%%%%        CONTROESEMPIO REGOLARITA' 

Finally it is interesting to remark that compactness and regularity
properties related to~$\Per_r$ can be problematic, or even fail,
at a small scale, also for minimizers.
To make a concrete example, we consider~$K>0$ and 
the function $g(x):=-K \chi_{B_r\setminus B_{r/2}}$.
We let
$$ {\mathcal{F}}_{K}(E):={\mathcal{F}_{r,g}}(E)= \Per_r(E)
-K\,{\mathcal{L}}^n\big( E\cap(B_r\setminus B_{r/2}
)\big).
$$
Then, minimizers are not necessarily smooth and sequences
of minimizers are not necessarily compact. Indeed, we have:

\begin{theorem}\label{FAIL:COM}
There exists~$C>0$, only depending on~$n$, for which the
following statement holds true.

Suppose that
\begin{equation}\label{K LARGE}
K\ge\frac{C}{r}.
\end{equation}
Then, there exists~$E_*\subseteq\R^n$ satisfying
$$ {\mathcal{F}}_K(E_*)\le {\mathcal{F}}_K(E)$$
for any bounded set~$E\subseteq\R^n$,
and such that~$\partial E_*$ is not
locally a continuous graph (and, in fact, can be ``arbitrarily bad''
inside~$B_{r/2}$).

Moreover, there exists
a sequence~$E_k\subseteq\R^n$ satisfying
$$ {\mathcal{F}}_K(E_k)\le {\mathcal{F}}_K(E)$$
for any bounded set~$E\subseteq\R^n$,
and such that~$\chi_{E_k}$ is not
precompact in~$L^1(B_r)$.
\end{theorem}

Given the negative result in Theorem~\ref{FAIL:COM}, we think that
it is an interesting problem to develop a regularity theory for
minimizers of~$\Per_r$ and of functionals such as~${\mathcal{F}_{r,g}}$.

%%%%%%%%%%%%%%%%%%%%%%%%%%%%%%%%%%%%%%%%%%%%%%%%%%%%%%%%%%%%%%%        PLANE LIKE MINIMIZERS

\subsection{Planelike minimizers in periodic media}
\label{Existence of planelike}

In the spirit of~\cite{MR1852978}, we recall the following definition:

\begin{definition}\label{PLANELIKE}  We say that a set $E\subseteq \R^n$ is planelike if, up to an
appropriate 
change of coordinates, there exists  $K>0$ such that  
\[E\supseteq \{(x_1, \dots, x_n) {\mbox{ s.t. }} 
x_n\leq 0\}\quad {\mbox{ and }}\quad \R^n\setminus E\supseteq
\{(x_1, \dots, x_n) {\mbox{ s.t. }} x_n\geq K\}.\]
\end{definition}

To state our result, we recall some notation. 
We say that a direction~$\omega\in S^{n-1}$ is rational
if the orthogonal space has maximal dimension over the integers, i.e.
\begin{equation}\label{RAT}\begin{split}&
{\mbox{there exist $K_1,\dots,K_{n-1}\in\Z^n$ which are linearly independent}}\\&
{\mbox{and such that~$\omega\cdot K_j=0$ for any~$j\in\{1,\dots,n-1\}$.}}\end{split}\end{equation}
Given a rational direction~$\omega\in S^{n-1}$, we say that a set~$E$
is~$\omega$-periodic if, for any~$k\in\Z^n$ with~$\omega\cdot k=0$,
we have that~$E+k=E$. Similarly, a function~$u:\R^n\to\R$
is said to be~$\omega$-periodic if, for any~$k\in\Z^n$ with~$\omega\cdot k=0$,
it holds that~$u(x+k)=u(x)$ for any~$x\in\R^n$.

Then, we state the following:

\begin{theorem}\label{PLANELIKE:TH}
There exist~$\eta\in(0,1)$ and~$M>1$, only depending on~$n$, such that the following result
holds true.
Let~$r\in(0,1)$, $g:\R^n\to\R$ be~$\Z^n$-periodic, with zero average in~$[0,1]^n$
and such that~$\|g\|_{L^\infty(\R^n)}\le\eta$.

Let~$\omega\in S^{n-1}$. Then, there exists~$E^*_\omega$
which is a Class~A minimizer for~$ {\mathcal{F}}_{r,g}$,
such that
\begin{equation}\label{PL:FOR}
\{\omega\cdot x\le -M\}\subseteq  E^*_\omega\subseteq
\{\omega\cdot x\le M\}
.\end{equation}

Moreover, if $\omega$ is rational, then $E^*_\omega$  is~$\omega$-periodic.
\end{theorem}

\subsection{Organization of the paper}\label{sec:org}

The rest of the paper is devoted to the proofs
of our main results.
Section~\ref{SRZ:2} contains the
proofs of  Proposition~\ref{TRIVIA}.
%Theorem~\ref{EQUIVALENCE} and
%Proposition~\ref{CLASSIFICATION:FAC}. 

The $\Gamma$-convergence results and the
compactness properties for
the functional $\mathcal{F}_{r,g}$, together with the proofs of
Theorem~\ref{SICOM:01},
Remarks \ref{NOCOM} and~\ref{REML01293dfas},
and
Theorem~\ref{promin}, are presented in Section~\ref{SD:3}.

%The proof of Proposition \ref{prominfunzioni}
%is contained in Section~\ref{sec:4}, while the
The proof of Theorems~\ref{DEBP}
%and~\ref{EXISTENCE}, and of Remarks~\ref{REM:NP}
%and~\ref{NOZ} are
is  contained in Section~\ref{P012wsPA}.

The characterizations
of Class~A minimizers in Proposition \ref{piani} and in
Theorem~\ref{DI1PR}   are dealt with
in Section~\ref{023peasq}.

We address the isoperimetric inequalities in Section~\ref{00232},
which contains the proofs of Lemma~\ref{ISOR},
Lemma~\ref{lemmatecnico},
Theorem~\ref{ISOPER},
Remark~\ref{FAIL},
Theorem~\ref{PW} and
Remark~\ref{NONPO}. 

The regularity and density estimates, with the proofs of
Theorems~\ref{DENSITY}
and~\ref{FAIL:COM}, are discussed in Section~\ref{91238erytfgdh}.

Finally, in Section~\ref{SA:1023}, we deal with the
construction of the
planelike minimizers in periodic media and we prove
Theorem~\ref{PLANELIKE:TH}.

%%%%%%%%%%%%%%%%%%%%%%%%%%%%%%%%%%%%%%%%%%%%%%%%%%%%%%%%%%%%%%%       DIMOSTRAZIONI RISULTATI IN 1.1. 

\section{Basic properties 
of minimizers of ${\Per_r}$  --
Proof of
Proposition~\ref{TRIVIA}} \label{SRZ:2}

We provide the proof of Proposition \ref{TRIVIA}, which
justifies our
definitions of local and Class~A  minimizers, given in Definition~\ref{DEFIN4}.

\begin{proof}[Proof of Proposition~\ref{TRIVIA}]
First of all, we claim that there exists
a universal~$\e>0$ such that
\begin{equation}\label{GEOGEO}
\left\{\left[\left(\R^n\setminus B_{1/\e}\left(\frac{e_n}{\e}\right) \right)
\cap \left(\R^n\setminus B_{1-\e}\right)
\right]\oplus B_1\right\}\cap B_\e(e_n)=\varnothing. 
\end{equation}
A picture can easily convince the   reader
about this simple geometric fact.

{F}rom now, we fix~$\e$ as in~\eqref{GEOGEO} and,
without loss of generality, we take~$\e\in\left(0,\frac12\right]$. As a matter of
fact, scaling~\eqref{GEOGEO}, we see that
\begin{equation*}
\left\{\left[\left(\R^n\setminus B_{r/\e}\left(\frac{r e_n}{\e}\right) \right)
\cap \left(\R^n\setminus B_{(1-\e)r}\right)
\right]\oplus B_r\right\}\cap B_{\e r}(r e_n)=\varnothing
\end{equation*}
and therefore
\begin{equation}\label{GEOGEO:BU}
\begin{split}&
{\mathcal{L}}^n\left(
\left\{\left[\left(\R^n\setminus B_{r/\e}\left(\frac{r e_n}{\e}\right) \right)
\cap \left(\R^n\setminus B_{(1-\e)r}\right)
\right]\oplus B_r\right\}\cap B_r
\right) \\&\qquad\le
{\mathcal{L}}^n\big( B_r\setminus B_{\e r}(r e_n)\big)<
{\mathcal{L}}^n(B_r).\end{split}
\end{equation}
Now we take~$E$ to be a Class~A minimizer
for~$\Per_r$
and we claim that
\begin{equation}\label{ISP01}
\begin{split}
&{\mbox{either there exists~$p\in\R^n$ such that~$B_{r/\e}(p)\subseteq E$,}}\\
&{\mbox{or there exists~$p\in\R^n$ such 
that~$B_{r/\e}(p)\subseteq \R^n\setminus E$.}}\end{split}
\end{equation}
The proof of~\eqref{ISP01}
is by contradiction: if not, any ball of radius~$r/\e$ contains
both points of~$E$ and of its complement, and so it contains
at least one point of~$\partial E$.

Let now~$M\ge10$ to be taken suitably large in the sequel and~$R:=Mr/\e$.
We consider $N$ disjoint balls of radius~$2r/\e$
contained in the ball~$B_R$,
and we observe that we can take~$N\ge \frac{c\,R^n}{(r/\e)^n}
=c M^n$, for some universal~$c>0$.
Let us call~$B_{2r/\e}(p_1),\dots,B_{2r/\e}(p_N)$ such balls.
We know that each ball~$B_{r/\e}(p_j)$ contains a point~$q_j\in\partial E$
and so~$(\partial E)\oplus B_r$
contains at least the balls~$B_r(q_j)$ which are disjoint and contained in~$B_R$.

Consequently,
\begin{equation}\label{9dufihvyewdjkasgcbicsddg}
2r\,\Per_r(E,B_R)\ge \sum_{j=1}^N {\mathcal{L}}^n (B_r(q_j))=
N {\mathcal{L}}^n (B_r)\ge \bar c M^n r^n,
\end{equation}
for some~$\bar c>0$.

Now we consider~$F:=E\cup B_{R-r}$.
Notice that~$\partial F\subseteq \R^n\setminus B_{R-r}$
and thus~$(\partial F)\oplus B_r\subseteq \R^n\setminus B_{R-2r}$.
This and
the minimality of~$E$ give that
$$ 2r\,\Per_r(E,B_R)\le 2r\,\Per_r(F,B_R)\le {\mathcal{L}}^n (B_{R}\setminus B_{R-2r})
\le C R^{n-1} r = \frac{CM^{n-1} r^{n}}{\e^{n-1}}.$$
{F}rom this and~\eqref{9dufihvyewdjkasgcbicsddg}
a contradiction easily follows by taking~$M$ appropriately large
(possibly also in dependence of the fixed~$\e$).
This completes the proof of~\eqref{ISP01}.

Now, from~\eqref{ISP01},
we can suppose that~$E$ contains 
a ball of radius~${r/\e}$ and we prove that~$E=\R^n$
(if instead~$\R^n\setminus E$ 
contains 
a ball of radius~$B_{r/\e}$, a similar argument would prove that~$E$
is void).

Sliding the ball till it touches the boundary of~$E$,
we find a ball of radius~$r/\e$ which lies in~$E$ and
whose boundary contains a point of~$\partial E$.
Therefore, up to a rigid motion, we can suppose that~$0\in\partial E$
and~$B_{r/\e}(re_n/\e)\subseteq E$.
We define
$$G:= E\cup B_{(1-\e)r}\supseteq B_{r/\e}(re_n/\e)\cup B_{(1-\e)r}.$$
Notice that
$$\partial G\subseteq \R^n\setminus\big(
B_{r/\e}(re_n/\e)\cup B_{(1-\e)r}\big)=
\big( \R^n\setminus B_{r/\e}(re_n/\e)\big)\cap
\big( \R^n\setminus B_{(1-\e)r}\big)$$
and so
$$(\partial G)\oplus B_r\subseteq \Big[ 
\big( \R^n\setminus B_{r/\e}(re_n/\e)\big)\cap
\big( \R^n\setminus B_{(1-\e)r}\big)\Big]\oplus B_r.$$
This, \eqref{GEOGEO:BU} and the minimality of~$E$ give that
\begin{equation}\label{9ischyeiwvdk2345fhd}
2r\,\Per_r(E,B_r)\le 2r\,\Per_r(G,B_r) < {\mathcal{L}}^n(B_r).
\end{equation}
On the other hand, since~$0\in\partial E$, we have that~$
(\partial E)\oplus B_r\supseteq B_r$
and therefore
$$ 2r\,\Per_r(E,B_r) \ge {\mathcal{L}}^n(B_r).$$
This is in contradiction with~\eqref{9ischyeiwvdk2345fhd}.
The proof of Proposition~\ref{TRIVIA}
is thus complete.
\end{proof}

\section{$\Gamma$-convergence results and compactness properties for  the functional $\mathcal{F}_{r,g}$ -- Proofs of
Theorem~\ref{SICOM:01},
Remarks \ref{NOCOM} and~\ref{REML01293dfas},
and
Theorem~\ref{promin}} \label{SD:3}

%%%%%%%%%%%%%%%%%%%%%
%%%%%%%%%%%%%%%%%%%%%
%%%%%%%%%%%%%%%%%%%%%        GAMMACONV R>0

We start with a preliminary result on the convergence of characteristic functions.

\begin{lemma}\label{lemchar}
Let $\Omega$ be an open subset of $\R^n$ and let $E_k$ be a sequence of sets such that 
$\chi_{E_k}$ converges to $u$ weakly in $L^1_{\rm loc}(\Omega)$. Then, letting 
$\Sigma:=\{x\in\Omega:\,u(x)\in (0,1)\}$, there holds
\begin{equation}\label{annabelve}
\chi_{E_k}\to u\quad \text{in }L^1(\Omega\setminus\Sigma).
\end{equation}
In particular, if $u$ is a characteristic function, then $\chi_{E_k}\to u$ in $L^1_{\rm loc}(\Omega)$. 
\end{lemma}

\begin{proof}
Without loss of generality we can assume that $\Omega$ is bounded.

Let $u_k:=\chi_{E_k}$.
Since~$0\le u_k\le1$, we have that
\begin{eqnarray*}&&
\lim_{k\to+\infty}
\int_{\Omega\setminus\Sigma}|u_k-u|=
\lim_{k\to+\infty}\left(
\int_{\Omega\cap\{u=1\}}(1-u_k)+
\int_{\Omega\cap\{u=0\}}u_k\right)\\&&\qquad=
\lim_{k\to+\infty} \left({\mathcal{L}}^n\big(\Omega\cap\{u=1\}\big)-
\int_{\Omega}u_k\chi_{\{u=1\}}+
\int_{\Omega}u_k\chi_{\{u=0\}}\right)\\&&\qquad=
{\mathcal{L}}^n\big(\Omega\cap\{u=1\}\big)-
\int_{\Omega}u\chi_{\{u=1\}}+
\int_{\Omega}u\chi_{\{u=0\}}=0,
\end{eqnarray*}
which proves~\eqref{annabelve}.
\end{proof}

\begin{remark}\label{lemcosy} Note that a consequence of the previous lemma is the following fact:
if  $F_k$ is a sequence of sets such that 
$\chi_{F_k}\to \chi_F$ in $L^1_{\rm loc}(\Omega)$, for some open set $\Omega$  and
for some $F\subset\R^n$ then   $\chi_{\lambda_k F_k}\to \chi_F$ in $L^1_{\rm loc}(\Omega)$ for all $\lambda_k\to 1$. 
Indeed it is sufficient to prove that   $\chi_{\lambda_kF_k}$ 
converges to $\chi_F$ weakly in $L^1_{\rm loc}(\Omega)$ and then apply the previous lemma.
\end{remark}

%\begin{proof}
%We first show that $\chi_{\widetilde F_k}$ 
%converges to $\chi_F$ weakly in $L^1_{\rm loc}(\Omega)$.
%For this, we let~$\varphi\in C_0(\Omega)$, and we have that
%\begin{eqnarray*}
%&&\lim_{k\to +\infty} \int_{\Omega\cap \widetilde F_k}
%\varphi(x)\,dx =
%\lim_{k\to +\infty} \int_{\widetilde F_k}
%\varphi(x)\,dx=
%\lim_{k\to +\infty} \lambda_k^n \int_{ F_k} 
%\varphi(\lambda_k y)\,dy \\&&\qquad\quad =\lim_{k\to +\infty} \int_{ F_k} 
%\varphi(\lambda_k y)\,dy
%= \int_{ F} \varphi(y)\,dy=\int_{\Omega\cap F}\varphi(y)\,dy.
%\end{eqnarray*}
%The thesis now follows by Lemma \ref{lemchar}.
%\end{proof}

We now provide the proof of the convergence result for $\Per_r$. 

\begin{proof}[Proof of Theorem~\ref{SICOM:01}]
First of all, we prove the claim in~$(1)$. 
For this, we observe that, for every $r>0$, it holds that
\begin{equation}\label{THIS1}
\mathcal{L}^n\Big(\partial \big((\partial E)\oplus B_r\big)\Big)=0.\end{equation}
This can be obtained e.g. as a consequence of the estimate 
proved in~\cite[Theorem 2]{MR3544940}: for all closed sets~$A$ 
and all $r>0$, it holds that
$\mathcal{H}^{n-1} (\partial (A\oplus B_r))\leq \frac{C}{r} 
\mathcal{L}^{n}\big((A\oplus B_r)\setminus A\big)$, where $C>0$ 
is a dimensional constant. Using this with~$A:=A_m=(\partial E)\cap B_m$,
for any fixed~$m\in\N$, we find that~$
\mathcal{H}^{n-1} (\partial (A_m\oplus B_r))<+\infty$, and therefore
%\begin{equation*}
%\mathcal{L}^{n} (\partial (A_m\oplus B_r))=0,
%\end{equation*}
%and so
\begin{equation}\label{9iwduef9igfwey3r00}
\mathcal{L}^{n} \left(\bigcup_{m\in\N}\partial (A_m\oplus B_r)\right)=0.
\end{equation}
We conclude observing that 
\begin{equation}\label{9iwduef9igfwey3r}
\partial \big((\partial E)\oplus B_r\big)
\subseteq\bigcup_{m\in\N}
\partial (A_m\oplus B_r).
\end{equation}
%To check this, let~$x\in \partial \big((\partial E)\oplus B_r\big)$.
%Then, there exists sequences~$\xi_k$ and~$\eta_k$ that converge to~$x$
%as $k\to+\infty$ and such that~$\xi_k\in(\partial E)\oplus B_r$
%and~$\eta_k\not\in(\partial E)\oplus B_r$. We take~$m$ so large that~$|x|+r+2<m$.
%Hence, for large~$k$, we can also suppose that
%\begin{equation}\label{POA1}
%|\xi_k|+r+1<m \quad{\mbox{ and }}\quad
%|\eta_k|+r+1<m.
%\end{equation}
%We notice that
%\begin{equation}\label{02erif}
%\eta_k\not\in A_m\oplus B_r.
%\end{equation}
%To check this, for a contradiction we suppose that~$\eta_k\in A_m\oplus B_r$.
%Then, there exists~$\widetilde\eta_k\in A_m\cap B_r(\eta_k)\subseteq(\partial E)\cap
%B_r(\eta_k)$. This would give that~$\eta_k\in (\partial E)\oplus B_r$,
%which is a contradiction, and so~\eqref{02erif} is established.
%
%We also observe that
%\begin{equation}\label{02erif:2}
%\xi_k\in A_m\oplus B_r.
%\end{equation}
%This follows by using~\eqref{POA1} to see that
%the closure of~$B_r(\xi_k)$ lies inside~$B_m$.
%
%Then, from \eqref{02erif} and~\eqref{02erif:2}, we obtain that~$x\in\partial
%\big(A_m\oplus B_r\big)$, as long as~$m$ is large enough.
%This proves~\eqref{9iwduef9igfwey3r}.
%Now, the claim in~\eqref{THIS1} follows from~\eqref{9iwduef9igfwey3r00}
%and~\eqref{9iwduef9igfwey3r}.

Using~\eqref{THIS1}, we see that $\chi_{(\partial E)\oplus B_{r_k}}\to 
\chi_{ (\partial E)\oplus B_r}$ a.e. in $\Omega$.

Hence, if~$\Omega$ is bounded, or if~$\Omega=\R^n$ and~$\partial E$ is bounded,
the assertion in~$(1)$ follows from
the Dominated Convergence Theorem.

To complete the proof of~$(1)$, we have only to consider the case in which~$\Omega=\R^n$
and~$\partial E$ is unbounded. In this case, we can take a sequence~$p_j\in\partial E$,
with~$|p_j|>2r+2+|p_{j-1}|$. In this way~$B_\rho(p_j)\cap B_\rho(p_i)$ is void when~$j\ne i$
and~$\rho\in (0,r+1)$, which gives that~${\mathcal{L}}^n\big( (\partial E)\oplus B_\rho\big)=+\infty$.
This says that, in this case,
$$ \Per_{r_k}(E, \Omega) =+\infty= \Per_{r}(E, \Omega),$$
and so~$(1)$ holds true.
\smallskip

Now, we prove the claim in~$(2)$. 
By $(1)$, we immediately deduce that 
\[\Gamma-\limsup_{r_k\to r} \Per_{r_k}(\cdot, \Omega)\leq \Per_r(\cdot, \Omega).\]
We are left to prove that, if $E_{k}\to E$ in $L^1_{\rm{loc}}(\Omega)$, then  
\[\liminf_{r_k\to r} \Per_{r_k}(E_{k}, \Omega)
\geq \Per_r(E, \Omega).\]
To see this, we observe that, 
if we set
\begin{equation}\label{def Erkwide}
\widetilde E_{k}:= \frac{r}{r_k}E_{k},
\qquad \widetilde \Omega_{k}:=\frac{r}{r_k}\Omega,
\end{equation}
then 
\begin{equation}\label{rischio}
\Per_{r_k}(E_{k}, \Omega)=\left(\frac{r_k}{r}\right)^{n-1} 
\Per_{r}(\widetilde E_{k}, \widetilde \Omega_{k}),
\end{equation}
and, recalling Remark \ref{lemcosy},
\begin{equation}\label{rescaled} 
\chi_{\widetilde E_{k}}\to \chi_E \text{ in $L^1_{\rm{loc}}(\Omega)$}.
\end{equation}  
Notice that, again by Remark \ref{lemcosy} applied with $F_k:=
\widetilde \Omega_k$ and 
$F:=\Omega$, we have that 
\begin{equation}\label{macchero} 
\vert \Per_{r}(\widetilde E_{k}, \widetilde \Omega_{k}) - 
\Per_{r}(\widetilde E_{k}, \Omega)\vert \le \frac{1}{2r}\,\vert \widetilde \Omega_{k}\Delta\Omega\vert
\to 0 \qquad{\rm as\ }k\to \infty.
\end{equation}
Therefore, by the lower semicontinuity of the functional~$\Per_r$, 
from \eqref{rischio}, \eqref{rescaled} and \eqref{macchero} we conclude that
\[\liminf_{r_k\to r} \Per_{r_k}(E_{r_k}, \Omega)=\liminf_{r_k\to r} \left(\frac{r_k}{r}\right)^{n-1}\Per_{r}(\widetilde E_{r_k}, \Omega)
\geq \Per_r(E, \Omega).\] 
This completes the proof of~$(2)$.
\smallskip

To prove~$(3)$,
we define~$\widetilde E_{r_k}$ as in~\eqref{def Erkwide}. Then 
$$ \ell =\liminf_{r_k\to r} \Per_{r}(\widetilde E_{r_k},\Omega).$$
%Thus, we can also fix~$\e>0$ (to be taken as small as we wish in the sequel)
%and suppose that, up to a subsequence,
%\begin{equation}\label{EPSBaO} 
%\Per_{r}(\widetilde E_{r_k},\Omega)\le \ell+\e.\end{equation}
Let $u_k:=\chi_{\widetilde E_{r_k}}$. 
Then, up to subsequences, 
$u_k$ converges to~$ u$ weakly-$\star$
in $L^\infty_{\rm{loc}}(\Omega)$ and 
then also $u_k$ converges to~$ u$  weakly in $L^1_{\rm{loc}}(\Omega)$,
with $u:\R^n\to [0,1]$. 
% Let $\Sigma=\{x\in \Omega | \  u(x)\in (0,1)\}$, where, as usual, 
% we identify the set with its points of density one. 
So, by the lower semicontinuity of the 
functional $\mathcal{E}_{1,\Omega}$, 
we have that \[\int_{\Omega} \osc_{B_r(x)} u\, dx\leq \liminf_{k}\int_{\Omega} \osc_{B_r(x)} 
u_k\, dx=2r \liminf_k\Per_r(\widetilde E_{r_k}, \Omega)=2r\ell,\]
which proves the third inequality in~\eqref{ZZ-1}.

So, by Lemma \ref{lemchar}, we have
\begin{equation}\label{DOVE}
{\mbox{$u_k\to u$ in  $L^1(\Omega\setminus\Sigma)$.}}\end{equation}
%which implies,
%%Then \eqref{DOVE} implies, 
%by the Dominated Convergence Theorem, the convergence
%result in~\eqref{ZZ-1}. 

Also, \eqref{DOVE} gives that, for all $x\in \Sigma\oplus B_r$, 
it holds that
\begin{equation}\label{9q2PA} 
\textrm{$(\partial E_k)\cap B_{r}(x)\neq \varnothing$ for $k$ large enough.}
\end{equation}
Indeed, if this is not true, then either~$B_r(x)\subseteq E_k$
or~$B_r(x)\subseteq\R^n\setminus E_k$ for infinitely many~$k$'s,
and so either~$u_k=1$ or~$u_k=0$
a.e. in~$B_r(x)$ for infinitely many~$k$.
This would imply that
either~$u=1$ or~$u=0$
a.e. in~$B_r(x)$, in contradiction with the fact that
$\Sigma\cap B_{r}(x)\neq \varnothing$,
and so~\eqref{9q2PA} is proved.

Using \eqref{9q2PA}, we get 
that $\osc_{B_r(x)} u_k\to 1$ for $x\in \Sigma\oplus B_r$, therefore 
\[\mathcal{L}^n((\Sigma\oplus B_r)\cap \Omega)\leq 
\liminf_{k} \int_{(\Sigma\oplus B_r)\cap \Omega} \osc_{B_r(x)} u_k\, dx
\leq 2r \liminf_{k}\Per_{r}(\widetilde E_{r_k},\Omega)=2r\ell,
\]
which completes the proof of~\eqref{ZZ-1}.

Then, the proof of Theorem~\ref{SICOM:01} is complete.
\end{proof}

We now exhibit the lack of compactness
that was claimed in Remark \ref{NOCOM}.

\begin{proof}[Proof of Remark \ref{NOCOM}] Let us take, for
example,  $r:=1$ and~$\Omega:=(-3,3)\times(0,1)^{n-1}\subseteq \R^{n}$.
Let also, for any~$k\ge1$,
$$ E_k := \bigcup_{j=-2^{k-1}}^{2^{k-1}}
\left( \frac{2j}{2^k},\frac{2j+1}{2^k}\right)\times (0,1)^{n-1}.$$
If~$\chi_{E_k}$ converged pointwise, it would also converge in~$L^1(\Omega)$,
due to the Dominated Convergence Theorem. But this is not the case,
since the norm in~$L^1(\Omega)$ of~$\chi_{E_k}-\chi_{E_{k+m}}$
is always bounded from below independently on~$k$ and~$m$.
\end{proof}

Now we present the pathological counterexample to Theorem~\ref{SICOM:01}, as stated in Remark~\ref{REML01293dfas}.

\begin{proof}[Proof of Remark~\ref{REML01293dfas}]
We take~$n\ge2$, $r>0$, $r_k:=r+\frac1k$, 
$\Omega:=\{x_n>0\}$ and~$E:=\{x_n < -r\}$.
In this way, for any~$\rho\ge r$,
$$ (\partial E)\oplus B_\rho = \{ x_n=-r\}\oplus B_{\rho}=
\big\{ x_n\in (-r-\rho, \,-r+\rho)\big\}.$$
Therefore
\begin{eqnarray*}&&
\big( (\partial E)\oplus B_r\big)\cap\Omega=
\{ x_n\in (-2r, 0)\}\cap\{ x_n>0\}=\varnothing,\\
{\mbox{and }}&&
 \big( (\partial E)\oplus B_{r_k}\big)\cap\Omega=
\left\{ x_n\in \left( -2r-\frac1k,\,
\frac1k
\right)\right\}\cap\{ x_n>0\}=\left\{ x_n\in \left( 0,\,
\frac1k
\right)\right\}.
\end{eqnarray*}
These considerations prove~\eqref{0q9weir9wurgrodighsdifhw}.
\end{proof}

Now we show that sequences of minimizers for the functional~${\mathcal{F}}_{r,g}$ produce a limit minimizer.

\begin{proof}[Proof of Theorem~\ref{promin}]
% We define $\Sigma:= \{x: u(x)\in (0,1)\}$, 
% always identifying $\Sigma$ with its points of density one. 
First of all we consider the case in which~$E=\{u=1\}$,
and we show that it is a local minimizer in $\Omega$ 
and that $g=0$ a.e. on $\Sigma$.

Observe that for all $ x\in A:=\left(\Sigma\cup \partial E\right)\oplus B_r$ it holds that
\begin{equation}\label{9q2PAn} 
\textrm{$(\partial E_k)\cap B_{r}(x)\neq \varnothing$ for $k$ large enough.}
\end{equation}
The proof of this fact is the same as the proof of \eqref{9q2PA} 
above (in the proof of Theorem~\ref{SICOM:01}). 

Fix $\Omega'\subseteq\Omega\ominus B_r$ and let
$$E_k^\star:= (E\cup(\Sigma\cap \{g<0\})\cap \Omega')\cup (E_k\setminus \Omega').$$
Observe that 
\begin{equation}\label{gek}
\int_{E_k } g\, dx= \int_{E_k^\star } g\, dx
+ \int_{(E_k\cap (\Sigma\cap\{ g\geq 0\}))\cap \Omega'} g\,dx 
-\int_{((\Sigma \cap\{g<0\}) \setminus E_k)\cap \Omega'} g\, dx 
+ \omega'_k
\end{equation}
with 
$$\omega'_k:=\int_{E_k\cap \{u=1\}\cap\Omega'} g \,dx-
\int_{\{u=1\}\cap\Omega'} g\,dx
+\int_{E_k\cap \{u=0\}\cap\Omega'}g\, dx.$$ 
Note that 
\[
\lim_{k\to+\infty} \omega'_k=
 \lim_{k\to+\infty}\int_{ \{u=1\}\cap\Omega'} g\chi_{E_k}\,dx
- \int_{\{u=1\}\cap\Omega'} g\,dx
+\lim_{k\to+\infty}\int_{\{u=0\}\cap\Omega'}g\chi_{E_k}\, dx =0.\] 
We define $\Sigma_k:= (E_k\Delta \{g<0\})\cap \Sigma\cap\Omega'$.
So \eqref{gek} reads
\begin{equation}\label{gek2}
\int_{E_k\cap \Omega'} g \, dx= \int_{E_k^\star\cap \Omega'} g\, dx
+ \int_{\Sigma_k} |g|\,dx + \omega'_k.\end{equation}
We also let 
$$C:= \left((E_k\setminus \overline{E_k^\star}) 
\cup (E_k^\star\setminus \overline{E_k}) \right)\cap \partial \Omega'.
$$
Notice that for all $x\in D:=(C\oplus B_r)\cap\Omega'$ there holds 
\begin{equation}\label{9q2PB} 
\textrm{$(\partial E_k)\cap B_{r}(x)\neq \varnothing$ for $k$ large enough.}
\end{equation}
To check this, we argue by contradiction and we suppose that, for instance, 
$B_r(x)\subseteq E_k$ for $k$ large enough. 
Then, $u_k=1$, and so~$u=1$ a.e. in $B_r(x)$, i.e.
$B_r(x)\cap \Omega'\subseteq  E\cap \Omega'$. Recalling 
that $B_r(x)\setminus \Omega'\subseteq E_k\setminus \Omega'$, this implies 
that~$B_r(x)\subseteq E_k^\star$. Accordingly, we have 
that~$B_r(x)\cap C=\varnothing$, and so~$x\not\in D$, against
our assumption. This proves~\eqref{9q2PB}.

Properties \eqref{9q2PAn} and \eqref{9q2PB} imply that 
\begin{equation}\label{ostregeta}
(A\cup D)\cap\Omega'\subseteq 
\big(((\partial E_k)\oplus B_r)\cap \Omega'\big)\cup O_k,
\end{equation}
for some~$O_k\subseteq\R^n$ with 
\begin{equation}\label{LIMITE}
\omega_k:=\frac1{2r}{\mathcal{L}}^n\big(O_k\big)\to 0\qquad \textrm{as $k\to +\infty$.}
\end{equation} 
By definition of $E_k^\star$, there holds
\begin{equation}\label{TGAH}
\big((\partial E_k^\star)\oplus B_r\big)\cap\Omega'\subseteq 
(A\cup D\cup ((\partial E_k)\oplus B_r))\cap\Omega',
\end{equation}
Therefore, 
from \eqref{ostregeta} and \eqref{TGAH} it follows that 
$$
{\mathcal{L}}^n\big(( (\partial E_k^\star)\oplus B_r)\cap \Omega'\big)
\le {\mathcal{L}}^n\big(( (\partial E_k)\oplus B_r)\cap \Omega'\big) + 2r
\omega_k
$$
and then 
\begin{equation}\label{perper}
{\Per_r}(E_k^\star, \Omega')\le {\Per_r}(E_k, \Omega')+\omega_k.
\end{equation}
From \eqref{perper} and \eqref{gek2} we get 
\[{\Per_r}(E_k^\star, \Omega')+ \int_{E_k^\star\cap \Omega'} g \,dx 
\le {\Per_r}(E_k, \Omega')+\int_{E_k\cap \Omega'} g\, dx
+\omega_k- \int_{\Sigma_k} |g|\,dx -\omega'_k.\]
Therefore, by minimality of $E_k$ we deduce that 
\[0=\lim_{k\to+\infty}  \int_{\Sigma_k} |g|\,dx=
\int_{\Sigma\cap \{g> 0\}}g u \,dx -\int_{\Sigma\cap \{g<0\} } g(1-u)\,dx.\]
This implies that $u=0$ on $\Sigma\cap \{g> 0\}$ and $u=1$ on $\Sigma\cap \{g< 0\}$, which, recalling the definition of $\Sigma$, implies that 
$\mathcal L^n(\Sigma\cap \{g> 0\})=0=\mathcal L^n(\Sigma\cap \{g<0\})$, so $g=0$ almost everywhere on $\Sigma$.  
If  $\mathcal L^n(\{ g=0\})=0$, we deduce that $\mathcal L^n(\Sigma)=0$, and then we conclude the strong convergence of $\chi_{E_k}$ to $\chi_E$.

Moreover, since $\Sigma_k\subseteq \Sigma$,  we conclude that
\begin{equation}\label{c1}{\Per_r}(E_k^\star, \Omega')+ \int_{E_k^\star\cap \Omega'} g \,dx 
\le {\Per_r}(E_k, \Omega')+\int_{E_k\cap \Omega'} g\, dx
+\omega_k -\omega'_k.\end{equation}

Let now $F$ be such that $F\Delta E\subseteq \Omega'\ominus B_r$. We define  
$$F_k:= (F\cap \Omega')\cup (E_k\setminus \Omega').$$
By construction 
$$
{\Per_r}(F_k,\Omega')-{\Per_r}(E_k^\star,\Omega')=
{\Per_r}(F,\Omega')-{\Per_r}(E,\Omega').
$$ 
Recalling \eqref{c1} we then get
\begin{eqnarray*}
{\Per_r}(F,\Omega')+\int_{F\cap \Omega'} g \,dx -{\Per_r}(E,\Omega') -\int_{E\cap \Omega'} g \,dx =  
{\Per_r}(F_k,\Omega')+\int_{F_k\cap \Omega'} g \,dx -{\Per_r}(E_k^\star,\Omega') -\int_{E_k^\star\cap \Omega'} g \,dx 
\\
\geq  {\Per_r}(F_k,\Omega')+\int_{F_k\cap \Omega'} g \,dx-{\Per_r}(E_k,\Omega')-\omega_k-\int_{E_k\cap \Omega'} g \,dx+\omega_k'\geq -\omega_k+\omega_k',
\end{eqnarray*} 
where the last inequality follows by the minimality of $E_k$.
Now we send~$k\to+\infty$ and we obtain that
$${\Per_r}(F,\Omega')+\int_{F\cap \Omega'} g \,dx -{\Per_r}(E,\Omega') -\int_{E\cap \Omega'} g \,dx \ge0,$$
thanks to~\eqref{LIMITE}.
This concludes the proof of the
local minimality of~$E=\{u=1\}$.

Now, let $E$ be any set such that~$\{u=1\}\subseteq E\subseteq
\Omega\setminus \{u=0\}$. Then we can 
define~$E_k^\star= ((E\cup(\Sigma\cap \{g<0\})\cap \Omega')
\cup (E_k\setminus \Omega')$ and repeat 
the same argument as above
(recalling that $g=0$ almost everywhere on $\Sigma$) 
to get that \eqref{c1} holds. 
The proof of Theorem~\ref{promin} is thus complete.
\end{proof}

\section{The Dirichlet problem -- Proof of Theorem~\ref{DEBP}
%and~\ref{EXISTENCE}, and of Remarks~\ref{REM:NP}
%and~\ref{NOZ}
} \label{P012wsPA}

\begin{proof}[Proof of Theorem~\ref{DEBP}]
Let~$E_k$ be a minimizing sequence.  Then, up to subsequences, 
$\chi_{E_k}\rightharpoonup u$ in~$L^1_{\rm{loc}}(\Omega)$,
with $u:\R^n\to [0,1]$. By the lower semicontinuity in $L^1$ of the functional $v\to \int_{\Omega} \osc_{B_r(x)} v\, dx$ proved in \cite{MR3187918} and the coarea formula,  we get \begin{align}\label{LA19eiydgfK}
\liminf_{k\to+\infty}\Per_r({E_k}, \Omega)
= \liminf_{k\to+\infty}\frac{1}{2r} 
\int_\Omega  \osc_{B_r(x)} \chi_{E_k} \,dx
\ge\frac{1}{2r} 
\int_\Omega  \osc_{B_r(x)} u \,dx=
\int_{0}^{1} \Per_r(\{u>s\}, \Omega) \,ds.\end{align}
Notice that
\begin{equation} \label{somega}
\int_{0}^{1} \Per_r(\{u>s\}, \Omega)\, ds\ge \Per_r(\{u>s_\Omega\}, \Omega),\end{equation}
for a suitable~$s_\Omega\in(0,1)$.
So, we define~$E:=\{u>s_\Omega\}$.
Since~$\chi_{E_k}$ does not depend on~$k$ outside~$\Omega'$,
we have that~$E=E_k$ outside~$\Omega'$ and thus it is an admissible competitor.
Then, \eqref{LA19eiydgfK} says that~$E$ is a minimizer
for~$\Per_r$, and this proves Theorem~\ref{DEBP}.
\end{proof}

\section{Class~A minimizers -- Proofs of Proposition \ref{piani} and of
Theorem~\ref{DI1PR} } \label{023peasq}
Now we prove the results about Class~A minimizers.
We start showing that half-spaces are Class~A minimizers for~$\Per_r$
in every dimension. 

\begin{proof}[Proof of Proposition \ref{piani}]
In this proof, we write $x=(x', x_n)\in \R^n$.
Up to translations and rotations, 
we can assume that $E=\{x\in\R^n {\mbox{ s.t. }} x_n<0\}$. 
We fix $B_R$ with $R>r$, and we consider $F\subseteq \R^N$ such that $F\Delta E\Subset B_R$. 
Let $C_R$ be the cylinder $\{x'\in\R^{n-1}{\mbox{ s.t. }} |x'| \leq R\} 
\times [-R, R]$, and 
observe that $\Per_r(E, C_R)= n\omega_n R^{n-1}$.

For any fixed~$x'\in\R^{n-1}$, let 
also~$\ell_{x'}=\{(x',x_n)\in\R^{n} {\mbox{ s.t. }} x_n\in\R\}$. We compute 
\[2r  \Per_r(F, C_R)= \int_{|x'|\leq R} \mathcal{H}^{1} 
((\partial F\oplus B_r )\cap\ell_{x'} )\,dx'
\geq 2r\int_{|x'|\leq R}\, dx' =2r \Per_r(E, C_R), \]
where we used the observation that~$\mathcal{H}^{1} 
((\partial F\oplus B_r )\cap \ell_{x'}) \geq 2r$, for every $x'$.
This proves Proposition \ref{piani}.
\end{proof} 

Now we characterize the Class~A 
minimizers of the nonlocal perimeter functional
in dimension~$1$ 
\begin{proof}[Proof of Theorem~\ref{DI1PR}]
Suppose that~$E\subseteq\R$ is a Class~A minimizer for~$\Per_r$.
Assume also that~$E\ne\varnothing$ and $E\ne\R$. Observe that this implies that  $E\not\subseteq (a,b)$ and  $\R^n \setminus E\not\subseteq (a,b)$ for every  $-\infty<a<b<+\infty$.
Indeed, if $E\subseteq (a,b)$ with $ -\infty<a<b<+\infty$, then  the empty set would be an admissible competitor for~$E$ in~$(a-r,b+r)$ and this would contradict the minimality of~$E$. Similarly for $\R^n\setminus E$. 
 
To conclude,   it is sufficient to show that $E$ is connected: 
\begin{equation}\label{SEGME}
{\mbox{if~$p$, $q\in E$ with~$p<q$, then~$(p,q)\subseteq E$.}}\end{equation}
We prove \eqref{SEGME} by contradiction.  

Assume it is not true, then there exists a point~$\beta\in (\partial E)\cap(p,q)$.
We define~$F:=E\cup(p,q)$ and we observe that~$F$ and~$E$ coincide
outside~$(p,q)$. Also, 
\begin{equation}\label{4.3bis}
{\mbox{$(\partial F)\cap (p,q)=\varnothing$ 
while~$(\partial E)\cap(p,q)\ni\beta$.}}\end{equation}
We also observe that
\begin{equation}\label{FEE} 
(\partial F)\setminus [p,q] = (\partial E)\setminus[p,q].\end{equation}
We claim that
\begin{equation}\label{FEE2} 
(\partial F)\setminus (p,q) \subseteq (\partial E)\setminus(p,q).\end{equation}
Indeed, if~$\zeta\in (\partial F)\setminus (p,q)$ then 
either~$\zeta\in(\partial F)\setminus[p,q]$, or~$\zeta\in\{p,q\}$.
If~$\zeta\in(\partial F)\setminus[p,q]$, then, by~\eqref{FEE},
we have that~$\zeta\in (\partial E)\setminus[p,q]\subseteq
(\partial E)\setminus(p,q)$, and we are done. 

Hence, we can focus on the case in which, for instance,~$\zeta=p$. 
Since~$F$ contains~$(p,q)$,
the fact that~$\zeta\in\partial F$ implies that there exists~$\zeta_k\in\R^n\setminus F$
with~$\zeta_k\le\zeta=p$. Then, by the definition of~$F$, we see 
that~$\xi_k\in\R^n\setminus E$. 
On the other hand, we know that~$\xi=p\in E$ (recall~\eqref{SEGME}). 
These observations
imply that~$\zeta=p\in\partial E$. 
This proves~\eqref{FEE2} also in this case.

{F}rom~\eqref{4.3bis} and~\eqref{FEE2} we get that
\begin{eqnarray*}&&
{\mathcal{L}}^n \Big(\big( (\partial E)\oplus(-r,r)\big)\cap (p-r,q+r)\Big)
-{\mathcal{L}}^n \Big(\big( (\partial F)\oplus(-r,r)\big)\cap (p-r,q+r)\Big)
\\ &=&
{\mathcal{L}}^n \Big(\big( (\partial E)\oplus(-r,r)\big)\cap (p,q)\Big)
-{\mathcal{L}}^n \Big(\big( (\partial F)\oplus(-r,r)\big)\cap (p,q)\Big)
\\ &&\qquad+
{\mathcal{L}}^n \Big(\big( (\partial E)\oplus(-r,r)\big)\cap \big((p-r,q+r)\setminus(p,q)\big)\Big)
\\ &&\qquad-{\mathcal{L}}^n \Big(\big( (\partial F)\oplus(-r,r)\big)\cap \big((p-r,q+r)\setminus(p,q)\big)\Big)
\\ &\ge& {\mathcal{L}}^n\big( (\beta-r,\beta+r)\big)
-{\mathcal{L}}^n\big( (0,r)\big)\\
&>&0.
\end{eqnarray*}
This implies that~$\Per_r(E,(p-r,q+r))>\Per_r(F,(p-r,q+r))$, which is against
minimality, and so the proof of~\eqref{SEGME} is completed.
\end{proof}
\section{Isoperimetric inequalities -- Proofs of Lemma~\ref{ISOR},
Lemma~\ref{lemmatecnico},
Theorem~\ref{ISOPER},
Remark~\ref{FAIL},
Theorem~\ref{PW} and
Remark~\ref{NONPO}} \label{00232}

Now, we deal with the isoperimetric problems.

\begin{proof}[Proof of Lemma~\ref{ISOR}]
First of all, we prove~(i). To this end, we remark that, without loss of generality,
we can suppose that~$\partial E$ is bounded
(if not, there would exist a sequence~$x_j\in\partial E$ such that~$|x_j|\ge j$
and~$|x_{j+1}-x_j|\ge 2r+1$, and thus~$\partial E\oplus B_r$
would contain the disjoint balls~$B_r(x_j)$, thus yielding that~$\Per_r(E)=+\infty$).
 
In addition, we notice that~$(\partial B_R)\oplus B_r=B_{R+r}\setminus B_{(R-r)^+}$ and therefore
$$ 2r\,\Per_r(B_R)=
{\mathcal{L}}^n \Big( (\partial B_R)\oplus B_r\Big)
={\mathcal{L}}^n ( B_{R+r}\setminus B_{(R-r)^+}).$$
By the Brunn-Minkowski Inequality (see e.g.~\cite{MR3155183}
or Theorem~4.1 in~\cite{MR1898210})
we have that
\begin{equation}\label{77:77}
\begin{split}& \Big( {\mathcal{L}}^n \big( E\oplus B_r\big)\Big)^{1/n}\geq 
\Big( {\mathcal{L}}^n (E)\Big)^{1/n}+\Big({\mathcal{L}}^n (B_r)\Big)^{1/n}
\\ &\qquad\qquad=
\Big( {\mathcal{L}}^n (B_R)\Big)^{1/n}+\Big({\mathcal{L}}^n (B_r)\Big)^{1/n}
=
\Big( {\mathcal{L}}^n (B_{R+r})\Big)^{1/n}
.\end{split}\end{equation}
As a consequence, we get
\begin{eqnarray}\label{RMuno}
{\mathcal{L}}^n \big( E\oplus B_r\big) - {\mathcal{L}}^n (E)
\ge {\mathcal{L}}^n \big( B_{R+r}\big) - {\mathcal{L}}^n (B_R).
\end{eqnarray}
We observe that if $R<r$, then $B_{R-r}=\emptyset$ and $E\ominus B_r=\emptyset$. 
Therefore \eqref{RMuno} implies (i).

On the other hand, if $R\geq r$, 
let us take
$\widetilde R\in[0,R]$ such that 
$$ {\mathcal{L}}^n \big( E\ominus B_r\big) =
{\mathcal{L}}^n \big( B_{\widetilde R}\big).$$
Also, recalling that $\big( E\ominus B_r\big)\oplus B_r\subseteq E$,
we have that 
$$
{\mathcal{L}}^n \big( (E\ominus B_r)\oplus B_r\big) \le 
{\mathcal{L}}^n (E) = {\mathcal{L}}^n (B_R).
$$ Accordingly,
applying again the Brunn-Minkowski Inequality we get that
\begin{eqnarray*}&&
{\mathcal{L}}^n (B_{R})^{1/n}\ge {\mathcal{L}}^n \big( (E\ominus B_r)\oplus B_r\big)^{1/n} 
\\&&\qquad\ge 
\Big( {\mathcal{L}}^n (E\ominus B_r)\Big)^{1/n}+\Big({\mathcal{L}}^n (B_r)\Big)^{1/n}
=
\Big( {\mathcal{L}}^n (B_{\widetilde R+r})\Big)^{1/n},
\end{eqnarray*}
which implies that $\widetilde R\le R-r$.

{F}rom this, we obtain that if $R\geq r$ 
\begin{equation}\begin{split}\label{RMdue}
{\mathcal{L}}^n (E) - 
{\mathcal{L}}^n \big( E\ominus B_r\big)\,&  =
{\mathcal{L}}^n (B_R)-{\mathcal{L}}^n (B_{\widetilde{R}})
\\ &
\ge {\mathcal{L}}^n (B_R) -
{\mathcal{L}}^n \big( B_{R-r}\big). 
\end{split}\end{equation}
  
Putting together \eqref{RMuno} and \eqref{RMdue} if $R\geq r$ we obtain
\begin{eqnarray*}&&
2r\,\Per_r(E)= {\mathcal{L}}^n \big( E\oplus B_r\big) - {\mathcal{L}}^n (E\ominus B_r)
\\&&\qquad\ge {\mathcal{L}}^n \big( B_{R+r}\big)-{\mathcal{L}}^n \big( B_{R-r}\big)
=2r\,\Per_r(B_R),
\end{eqnarray*}
thus proving~(i).

Now, we prove~(ii). For this, we observe that if equality holds,
then all the previous equalities hold true with equal sign.
In particular,
formula~\eqref{77:77} would give that
$$ \Big( {\mathcal{L}}^n \big( E\oplus B_r\big)\Big)^{1/n}= 
\Big( {\mathcal{L}}^n (E)\Big)^{1/n}+\Big({\mathcal{L}}^n (B_r)\Big)^{1/n}.$$
Hence (see e.g.
page~363 in~\cite{MR1898210}), since
equality  holds  in  the
Brunn-Minkowski inequality
if and only if
the two sets are homothetic
convex  bodies (up to removing sets of measure zero),
we have that~$E=B_{\lambda R}(p)\setminus{\mathcal{N}}$,
for some set~${\mathcal{N}}$ of null measure, some~$p\in\R^n$ and some~$\lambda>0$.
Since
$$ {\mathcal{L}}^n(B_R)=
{\mathcal{L}}^n(E)={\mathcal{L}}^n\big( B_{\lambda R}(p)\setminus{\mathcal{N}}\big)=
\lambda^n {\mathcal{L}}^n(B_R),$$
we obtain that~$\lambda=1$, which establishes~(ii).
\end{proof}

Having settled the global isoperimetric problem, we now deal with
the proof of the relative isoperimetric inequality.
First of all we give the proof of the technical lemma.

\begin{proof}[Proof of Lemma~\ref{lemmatecnico}] 
We consider a partition of~$\R^n$ into adjacent cubes of side~$\frac{r_k}{4\sqrt{n}}$
(hence, the diameter of each cube is~$\frac{r_k}4$).
These cubes will be denoted by~$\{ Q_j\}_{j\in\N}$.
For any~$k\in\N$, we set
\begin{equation}\label{7810284pa} 
I_k:=\{ j\in \N {\mbox{ s.t. }} Q_j\cap E_k\ne\varnothing\}.
\end{equation}
Let also
$$ \widehat E_k:=\bigcup_{ j \in I_k} Q_j.$$
Notice that~\eqref{0001:INCLU} is obvious in this setting.
We now prove~\eqref{0001:BIS}. For this,
we say that~$Q_j$ is a $k$-boundary cube
if~$j\in I_k$ and there exists a cube~$Q_i$ that is adjacent
to~$Q_j$ with~$i\not\in I_k$. We let~$\beta_k$ be the number
of $k$-boundary cubes which intersect~$\Omega_k$.

We remark that
\begin{equation}\label{8qyuwdgfh2trefduwfyuifgyewisfg10001}
\Per( \widehat E_k,\,
\Omega_k )\le C\beta_k r_k^{n-1},
\end{equation}
for some~$C>0$. We also claim that
\begin{equation}\label{781293847dhu}
\beta_k\le \frac{C\, \Per_{r_k}(E_k,\Omega_k)}{r_k^{n-1}}.
\end{equation}
up to renaming~$C>0$. To this end,
let~$Q_j$ be a $k$-boundary cube
and~$Q_i$ be its adjacent cube with~$j\in I_k$ and~$i\not \in I_k$.
Thus, by~\eqref{7810284pa}, there exists~$p_{j,k}\in Q_j\cap E_k$
and~$p_{i,k}\in Q_i\setminus E_k$. Consequently,
we find a point~$p^\star_k\in \partial E_k$ which lies at distance
at most~$r_k/4$ from~$Q_j$. Therefore
\begin{equation} \label{Piop}
(\partial E_k)\oplus B_{r_k}\supseteq
B_{r_k}(p_k^\star)\supseteq Q_j\oplus B_{\frac{r_k}{100}}.\end{equation}
In addition, if~$Q_j$ intersects~$\Omega_k$, it follows from~\eqref{LIPOM}
that (for large~$k$)
$$ {\mathcal{L}}^n \big( (
Q_j\oplus B_{\frac{r_k}{100}}) \cap \Omega_k
\big)\ge \frac{r_k^n}{C},$$
for some~$C>0$.
Hence, if~$Q_j^\star$ denotes the dilation of~$Q_j$ by a factor~$2$ with respect to
its center, we have that~$Q^\star_j\supseteq Q_j\oplus B_{\frac{r_k}{100}}$ and
$$ {\mathcal{L}}^n \big(
(Q_j\oplus B_{\frac{r_k}{100}}) \cap \Omega_k\cap Q_j^\star
\big)
={\mathcal{L}}^n \big((
Q_j\oplus B_{\frac{r_k}{100}} )\cap \Omega_k
\big)
\ge \frac{r_k^n}{C}.$$
This and~\eqref{Piop} give
that
\begin{equation}\label{OVG} {\mathcal{L}}^n \Big(
\big( (\partial E_k)\oplus B_{r_k}\big)\cap\Omega_k
\cap Q_j^\star \Big)\ge 
\frac{ r_k^n}C.\end{equation}
Our goal is now to sum up~\eqref{OVG}
for all the indices~$j$ for which~$Q_j$ is a boundary cube that intersects~$\Omega_k$.
Notice that the family~$\{
Q_j^\star \}_{j\in\N}$ is overlapping (differently from the original
nonoverlapping family~$\{Q_j\}_{j\in\N}$), but the number of overlappings
is finite, say bounded by some~$C^\star>0$.
Hence,
since~\eqref{OVG} is valid for any $k$-boundary cube~$Q_j$
which intersect~$\Omega_k$, summing up~\eqref{OVG} over the indices~$j$
gives that
\begin{eqnarray*}
&& C^\star\,{\mathcal{L}}^n \Big(
\big( (\partial E_k)\oplus B_{r_k}\big)\cap\Omega_k \Big)\ge 
\sum_{j\in\N}
{\mathcal{L}}^n \Big(
\big( (\partial E_k)\oplus B_{r_k}\big)\cap\Omega_k
\cap Q_j^\star \Big)\\
&&\qquad\ge
\sum_{ \small{\mbox{$k$-boundary cube~$Q_j$}}\atop{\mbox{which intersect~$\Omega_k$}}}
{\mathcal{L}}^n \Big(
\big( (\partial E_k)\oplus B_{r_k}\big)\cap\Omega_k
\cap Q_j^\star \Big)
\ge
\frac{ \beta_k r_k^n}C\end{eqnarray*}
and thus
$$ C^\star\,\Per_{r_k}(E_k,\Omega_k)\ge \frac{ \beta_k r_k^{n-1}}C,$$
that establishes~\eqref{781293847dhu}, up to renaming constants.

{F}rom~\eqref{8qyuwdgfh2trefduwfyuifgyewisfg10001} and~\eqref{781293847dhu}
it follows that~\eqref{0001:BIS} holds true, as desired.

In addition, from~\eqref{UNIKA} and~\eqref{0001:BIS},
we obtain a uniform bound for~$\Per( \widehat E_k,\,
\Omega_k )$ and thus on~$
\Per( \widehat E_k,\,
\Omega )$, so by compactness, up to a subsequence we
have that there exists $E\subseteq\R^n$ for which 
\begin{equation}\label{convl1}
\chi_{\widehat E_k}\to\chi_E \quad{\mbox{ in }}L^1(\Omega).
\end{equation}

Now we prove~\eqref{0002:BIS}. For this, let
\begin{eqnarray*} &&
J_k:=\big\{ j\in I_k {\mbox{ s.t. }} Q_j\cap \Omega_k\ne\varnothing
{\mbox{ and }} Q_j\setminus E_k\ne\varnothing\big\}\\
{\mbox{and }}&& H_k:=\bigcup_{j\in J_k}Q_j.\end{eqnarray*}
Notice that
\begin{equation}\label{CF91}
(\widehat E_k \setminus E_k)\cap
\Omega_k \subseteq H_k.
\end{equation}
To check this, let~$x\in (\widehat E_k \setminus E_k)\cap
\Omega_k$. Then, there exists~$j\in I_k$ such that~$x\in Q_j$.
Notice that~$x\in Q_j\setminus E_k$ and~$x\in Q_j\cap
\Omega_k$, which means that~$j\in J_k$, and so~$x\in H_k$,
thus proving~\eqref{CF91}.

Now we prove that
\begin{equation}\label{77-CARDI}
{\mbox{the cardinality of~$J_k$ is
bounded by }}\,\frac{C\,\Per_{r_k}(E_k,\Omega_k)}{r_k^{n-1}},
\end{equation}
Indeed, if~$j\in J_k$, then also~$j\in I_k$, therefore~$Q_j\cap E_k\ne\varnothing$
and also~$Q_j\setminus E_k\ne\varnothing$. Hence there exists~$x_{j,k}\in Q_j\cap(\partial E_k)$.
Notice that
\begin{equation}\label{78XCAY}
B_{r_k}(x_{j,k})\supseteq Q_j\oplus B_{\frac{r_k}{100}}
\end{equation}
Also, $Q_j\cap \Omega_k\ne\varnothing$. Consequently,
making use of~\eqref{LIPOM} and~\eqref{78XCAY}, we see that
\begin{eqnarray*}
&&{\mathcal{L}}^n \Big(\big((\partial E_k)\oplus B_{r_k}\big)\cap\Omega_k\cap ( Q_j\oplus B_{\frac{r_k}{100}})\Big)\ge
{\mathcal{L}}^n \big( B_{r_k}(x_{j,k})\cap\Omega_k\cap (Q_j\oplus B_{\frac{r_k}{100}})\big)\\&&\qquad
\ge {\mathcal{L}}^n \big( (Q_j\oplus B_{\frac{r_k}{100}})\cap\Omega_k\big)
\ge \frac{r_k^n}{C},
\end{eqnarray*}
up to renaming~$C>0$. Since this is valid for any~$j\in J_k$ and there is a finite number of overlaps
between different~$Q_j\oplus B_{\frac{r_k}{100}}$, we conclude that
$$ {\mathcal{L}}^n \Big(\big((\partial E_k)\oplus B_{r_k}\big)\cap\Omega_k\Big)\ge
\frac{r_k^n\; \# J_k}{C},$$
up to renaming~$C>0$
that implies~\eqref{77-CARDI}.

Now, in view of~\eqref{CF91} and~\eqref{77-CARDI}, we find that
\begin{eqnarray*}&& {\mathcal{L}}^n\big( (\widehat E_k \setminus E_k)\cap
\Omega_k\big)\le{\mathcal{L}}^n( H_k)\le\sum_{j\in J_k}
{\mathcal{L}}^n(Q_j)\\&&\qquad\le C\,r_k^n\,\# J_k\le C\,r_k\,\Per_{r_k}(E_k,\Omega_k).\end{eqnarray*}
This implies~\eqref{0002:BIS}.

Finally, \eqref{UNIKArk}, \eqref{UNIKA} and~\eqref{0002:BIS} give that
$$ \chi_{\widehat E_k}-\chi_{E_k}\to 0 \quad{\mbox{ in }}L^1(\Omega),$$
and this, combined with~\eqref{convl1}, 
implies~\eqref{nuovaconv}, as desired.
\end{proof}

With this, we can now complete the proof of Theorem~\ref{ISOPER}.

\begin{proof}[Proof of Theorem~\ref{ISOPER}] 
First of all we consider the case in which $R<r\leq \lambda R$ for $\lambda>1$. 
By assumption we know that ${\mathcal{L}}^n (E\cap B_{R})\leq \frac{1}{2} {\mathcal{L}}^n (B_{R})=\frac12 \omega_n R^n$. So either ${\mathcal{L}}^n (E\cap B_{R})=0$  and there is nothing to prove, or
${\mathcal{L}}^n (E\cap B_{R})\not=0$. In this case $\partial E\cap B_R\not=\emptyset$, and so $\Per_r(E, B_R)\geq  \frac{1}{2r} \frac{{\mathcal{L}}^n (B_{R})}{3}=
\frac{\omega_n R^n}{6r} $.
Summarizing we get 
\[ \Per_r(E, B_R)\geq \frac{\omega_n R^n}{6r}\geq  \frac{\omega_n R^{n-1}}{6\lambda} \geq \frac{\omega_n}{6\lambda} \left(\frac{2{\mathcal{L}}^n (E\cap B_{R})}{\omega_n}\right)^{\frac{n-1}{n}}=\frac{C}{\lambda}\left( {\mathcal{L}}^n (E\cap B_{R})\right)^{\frac{n-1}{n}}. \]

We consider now the case $\lambda=1$, so $r\leq R$, and we argue by contradiction.
If~\eqref{78:89ooo} were not true, recalling also~\eqref{MAGGIORE}
and~\eqref{178:89ooo},
we would infer that there exist sequences
\begin{equation}\label{89iwueeuigifuegifgedufg}
R_k\ge r_k>0\end{equation}
and~$E_k\subseteq\R^n$ such that
\begin{equation}\label{6767}\begin{split}
&
\frac{ {\mathcal{L}}^n (E_k\cap B_{R_k})}{
{\mathcal{L}}^n (B_{R_k}) }\le \frac12\\
{\mbox{and }}\;&\Big(
{\mathcal{L}}^n (E_k\cap B_{R_k} )\Big)^{\frac{n-1}n}> k\,\Per_{r_k}(E_k,B_{R_k}).
\end{split}\end{equation}
We define~$\lambda_k:=\big( {\mathcal{L}}^n(E_k\cap B_{R_k})\big)^{-\frac1n}$,
$\widetilde E_k:= \lambda_k E_k$,
$\widetilde r_k:= \lambda_k r_k$ and~$\widetilde R_k=\lambda_k R_k$.
With this scaling, we have that
\begin{equation}\label{9090y823ewef3sdf342r}
{\mathcal{L}}^n (\widetilde E_k\cap B_{\widetilde R_k})=
{\mathcal{L}}^n \big(\lambda_k ( E_k\cap B_{R_k})\big)=
\lambda_k^n {\mathcal{L}}^n( E_k\cap B_{R_k})=1.
\end{equation}
Moreover,
$$ \Per_{\widetilde r_k}(\widetilde E_k, B_{\widetilde R_k})=
\Per_{\lambda r_k}(\lambda_k E_k, \lambda_k B_{R_k})=\lambda_k^{n-1}
\Per_{r_k}(E_k, B_{R_k}).
$$
Therefore~\eqref{6767} becomes
\begin{equation}\label{6768} 
{\mathcal{L}}^n (B_{\widetilde R_k})\ge2\qquad 
{\mbox{and }}\qquad 
\Per_{\widetilde r_k}(\widetilde E_k,B_{\widetilde R_k}
) =\lambda_k^{n-1}
\Per_{r_k}(E_k, B_{R_k}) <\frac1k \lambda_k^{n-1}\Big({\mathcal{L}}^n (E_k\cap B_{R_k} )\Big)^{\frac{n-1}n}=\frac{1}{k}.
\end{equation}
Thanks to the first inequality in~\eqref{6768}, setting
$$ \widetilde R_o:=\liminf_{k\to+\infty}\widetilde R_k,$$
we have that~$R_o\in (0,+\infty]$ and
\begin{equation}\label{gy834bt48723vctr73nygyuerg}
{\mathcal{L}}^n (B_{\widetilde R_o})\ge2.\end{equation}
Here, the obvious notation~$B_{\widetilde R_o}=\R^n$ if~$R_o=+\infty$
has been used.

Now we claim that
\begin{equation}\label{tp01}
\widetilde r_k\to0.
\end{equation}
For this, we observe that~$\widetilde
R_k\ge \widetilde r_k$,
thanks to~\eqref{89iwueeuigifuegifgedufg}.

In addition,
$$ {\mathcal{L}}^n (\widetilde E_k\cap B_{\widetilde R_k})= 1<2\le
{\mathcal{L}}^n (B_{\widetilde R_k}),$$
thanks to~\eqref{9090y823ewef3sdf342r}
and~\eqref{6768}. Therefore both~$\widetilde E_k\cap B_{\widetilde R_k}$
and~$B_{\widetilde R_k}\setminus\widetilde E_k$
are nonvoid, and so there exists~$p_k\in(\partial \widetilde E_k)\cap B_{\widetilde R_k}$.
Accordingly,
$$ \Per_{\widetilde r_k}(\widetilde E_k, B_{\widetilde R_k})\ge\frac{1}{2\widetilde r_k}
{\mathcal{L}}^n ( B_{\widetilde r_k}(p_k)\cap B_{\widetilde R_k} )\ge
\frac{c\,\min\{ \widetilde r_k^n,\;\widetilde R_k^n\} }{\widetilde r_k} = c\,\widetilde r_k^{n-1},$$
for some~$c>0$.
{F}rom this
and~\eqref{6768} we deduce that
$$  c\,\widetilde r_k^{n-1}\le
\Per_{\widetilde r_k}(\widetilde E_k,B_{\widetilde R_k}
) <\frac1k,
$$
which proves~\eqref{tp01}, as desired.

In light of~\eqref{tp01}, we can now exploit
Lemma \ref{lemmatecnico} (with 
$\Omega_k:=B_{\widetilde R_k}$ and~$\Omega:=\cap_k B_{\widetilde R_k}$,
which is nontrivial thanks to~\eqref{gy834bt48723vctr73nygyuerg}).
In particular, from~\eqref{0001:INCLU} and~\eqref{0001:BIS},
we know that there exists~$\widehat E_k\subseteq\R^n$ such that
\begin{equation}\label{USA:1} \widehat E_k\supseteq \widetilde E_k\end{equation}
and
$$
\Per( \widehat E_k,\,
B_{\widetilde R_k} )\le C\,\Per_{\widetilde r_k}(\widetilde E_k,B_{\widetilde R_k})
.$$
Therefore, recalling~\eqref{6768},
\begin{equation}\label{USA:2} \Per( \widehat E_k,\,
B_{\widetilde R_k} )\le \frac{C}{k}.\end{equation}
Moreover, using~\eqref{0002:BIS},
\begin{equation}\label{9090y823ewef3sdf342rBIS} \int_{B_{\widetilde R_k} }
|\chi_{\widetilde E_k}-\chi_{\widehat E_k}|\,dx\le
C\,\widetilde r_k\,\Per_{\widetilde r_k}(\widetilde E_k,B_{\widetilde R_k})
\le \frac{C\widetilde r_k}{k }.
\end{equation}
Using~\eqref{9090y823ewef3sdf342r}
and~\eqref{9090y823ewef3sdf342rBIS}, we see that
\begin{eqnarray*} {\mathcal{L}}^n (\widehat E_k\cap B_{\widetilde R_k})&\le&
{\mathcal{L}}^n (\widetilde E_k\cap B_{\widetilde R_k})
+{\mathcal{L}}^n \big((\widehat E_k\setminus\widetilde E_k)\cap B_{\widetilde R_k}\big)\\
&\le& 1+ \frac{C\widetilde r_k}{k}.
\end{eqnarray*}
This and~\eqref{gy834bt48723vctr73nygyuerg} imply that
\begin{equation}\label{9j2usfgwe8v19urioqa} \lim_{k\to+\infty}
\frac{ {\mathcal{L}}^n (\widehat E_k\cap B_{\widetilde R_k}) }{
{\mathcal{L}}^n (B_{\widetilde R_k}) }\le
\frac{1}{ {\mathcal{L}}^n (B_{\widetilde R_o}) }\le\frac12.\end{equation}
So, we can assume that, for large~$k$,
$$ \frac{ {\mathcal{L}}^n (\widehat E_k\cap B_{\widetilde R_k}) }{
{\mathcal{L}}^n (B_{\widetilde R_k}) }\le\frac34,$$
hence we can apply the classical relative isoperimetric inequality and find that
$$ \Big(
{\mathcal{L}}^n (\widehat E_k\cap B_{\widetilde R_k})\Big)^{\frac{n-1}{n}} 
\le C\,\Per(\widehat E_k,B_{\widetilde R_k}).$$
Consequently, recalling~\eqref{USA:1} and~\eqref{USA:2},
$$\Big(
{\mathcal{L}}^n (\widetilde E_k\cap B_{\widetilde R_k})\Big)^{\frac{n-1}{n}} \le
\frac{C}{k }.$$ {F}rom this,
sending~$k\to+\infty$ and recalling~\eqref{9090y823ewef3sdf342r},
we obtain a contradiction that proves Theorem~\ref{ISOPER}.
\end{proof}

Now we check that ~\eqref{78:89ooo} cannot hold with a constant independent of $\lambda$.

\begin{proof}[Proof of Remark~\ref{FAIL}]
As an example, let~$n=2$, $R=100$ and~$E:=B_1$.
Notice that~\eqref{178:89ooo} is satisfied, but~\eqref{78:89ooo}
cannot be true for arbitrarily large~$r$ for some constant $C$ independent of the rate $\frac{r}{R}$.
Indeed, we have that~$\partial E\subseteq B_{100}$, hence
$$ \big( (\partial E)\oplus B_r\big)\cap B_R\subseteq
B_{ 100+r}\cap B_R = B_{100}.$$
As a consequence, if $r$ is sufficiently large,
$$ \Per_r (E,B_R)\le\frac{1}{2r} {\mathcal{L}}^n (B_{100})
< \frac1C \,\Big({\mathcal{L}}^n(B_1)\Big)^{\frac{n-1}n}=
\frac1C \,\Big({\mathcal{L}}^n(E\cap B_R)\Big)^{\frac{n-1}n},$$
thus violating~\eqref{78:89ooo}.
\end{proof}

Now, we can provide the easy proof of the
Poincar\'e-Wirtinger inequality in 
Theorem~\ref{PW}:

\begin{proof}[Proof of Theorem~\ref{PW}]
Up to a vertical translation, we may and do suppose that
\begin{equation}\label{ZE:AV}
{\mbox{$u$
has zero average in~$B_R$.}}\end{equation}
Moreover,
$$ {\mathcal{L}}^n (\{u>0\}\cap B_R)+
{\mathcal{L}}^n (\{u<0\}\cap B_R) \le
{\mathcal{L}}^n (B_R).$$
Thus, possibly exchanging~$u$ with~$-u$, we may and do suppose that
\begin{equation}\label{VE:CO}
{\mathcal{L}}^n (\{u>0\}\cap B_R) \le \frac{ {\mathcal{L}}^n (B_R) }{2}.
\end{equation}
Let also~$u^+:=\max\{u,0\}$. Then, using~\eqref{ZE:AV}, we see that
\begin{eqnarray*}&& \int_{B_R} |u| =\int_{B_R\cap\{u>0\}} u-
\int_{B_R\cap\{u<0\}} u\\&&\qquad=2\int_{B_R\cap\{u>0\}} u-
\int_{B_R} u = 2\int_{B_R\cap\{u>0\}} u^+ - 0.\end{eqnarray*}
Hence, integrating with respect to the distribution function
(see e.g. Theorem~5.51 in~\cite{MR3381284}),
we have that
\begin{equation}\label{DAQUI0}
\int_{B_R} |u|=2
\int_{B_R\cap\{u>0\}} u^+ =
2 \int_0^{+\infty} {\mathcal{L}}^n (\{u^+>s\}\cap B_R)\,ds.
\end{equation}
In addition, from~\eqref{VE:CO}, for any~$s\ge0$ we have that
$$ {\mathcal{L}}^n (\{u^+>s\}\cap B_R) 
=
{\mathcal{L}}^n (\{u>s\}\cap B_R) 
\le {\mathcal{L}}^n (\{u>0\}\cap B_R)
\le \frac{ {\mathcal{L}}^n (B_R) }{2}.$$
Consequently, we can exploit our relative isoperimetric inequality
in Theorem~\ref{ISOPER} with~$E:=\{u^+>s\}$ and conclude that, for any~$s\ge0$,
$$ \Big(
{\mathcal{L}}^n (\{u^+>s\}\cap B_R)\Big)^{\frac{n-1}{n}}
\le C\lambda\,\Per_r(\{u^+>s\},B_R),$$
for some~$C>0$. Multiplying this estimate by~$
\Big( {\mathcal{L}}^n (\{u^+>s\}\cap B_R)\Big)^{\frac{1}{n}}$,
we obtain that, for any~$s\ge0$,
\begin{eqnarray*}
{\mathcal{L}}^n (\{u^+>s\}\cap B_R)&\le&
  C\lambda\,\Per_r(\{u^+>s\},B_R)\, \Big( {\mathcal{L}}^n (\{u^+>s\}\cap B_R)\Big)^{\frac{1}{n}}
\\ &\le&   C\lambda\ R\, \Per_r(\{u^+>s\},B_R),
\end{eqnarray*}
up to renaming~$C>0$. Accordingly,
\begin{eqnarray*}
&& \int_0^{+\infty} {\mathcal{L}}^n (\{u^+>s\}\cap B_R)\,ds\le
 C\lambda R\, \int_0^{+\infty} \Per_r(\{u^+>s\},B_R)\,ds
\\&&\qquad \le  C \lambda R\, \int_\R \Per_r(\{u^+>s\},B_R)\,ds=\frac{C\lambda R}{r}\int_{B_R}\osc_{B_r(x)} u,
\end{eqnarray*}
thanks to the coarea formula in~\eqref{COAREA}.
Hence, recalling~\eqref{DAQUI0}, we conclude that
$$ \int_{B_R} |u|\le\frac{2C\lambda R}{r}\int_{B_R}\osc_{B_r(x)} u,$$
which is the desired result, up to renaming constants.
\end{proof}

Now we check that Theorem~\ref{PW} does not hold in general when~$r>R$ with a constant independent of the rate $\frac{r}{R}$:

\begin{proof}[Proof of Remark~\ref{NONPO}]
Let~$R=1$ and
$$ u(x):= \left\{ \begin{matrix} 1 & {\mbox{ if }} x>0, \cr
0 & {\mbox{ if }} x=0,\cr
-1 & {\mbox{ if }} x<0.\end{matrix}\right.$$
Notice that~$u$ has zero average and its oscillation is always bounded by~$2$.
Therefore, if~$r$ is large enough,
$$ \frac{CR}{r} \int_{B_R} \osc_{B_r(x)} u\,dx\le \frac{C}{r}
2 {\mathcal{L}}^n(B_1) < {\mathcal{L}}^n(B_1)=\int_{B_R}
\big|u - \langle u \rangle_R\big|,$$
which violates~\eqref{89w1:NOATo}.
\end{proof}

%%%%%%%%%%%%%%%%%%%%%%%%%%%%%%%%%%%%%%%%%%%%%%%%%%%%%%%%%%%%%%%%%%%%%%%%%%%%%%%%%%%%%%
%%%%%%%%%%%%%%%%%%%%%%%%%%%%%%%%%%%%%% %%%%%%%%%%%%%%%%%%%%%%%%%%%%%%%%%%%%%%}%]
%%%%%%%%%%%%%%%%%%%%%%%%%%%%%%%%%%%%%%%%%%%%%%%%%%%%%%%%%%%%%%%%%%%%%%%%%%%%%%%%%%%%

\section{Regularity issues and density estimates -- Proofs of 
Theorems~\ref{DENSITY}
and~\ref{FAIL:COM}} \label{91238erytfgdh}

In this section we prove the nonlocal density estimates
in Theorem~\ref{DENSITY}:

\begin{proof}[Proof of Theorem~\ref{DENSITY}]
We set~$f(R):={\mathcal{L}}^n(E\cap B_R)$.
We notice that if~$R-r\ge r$ and~$f(R-r)\le \frac{{\mathcal{L}}^n(B_R) }{2}$, then
we can apply the relative isoperimetric inequality in
Theorem~\ref{ISOPER} and obtain that
\begin{equation}\label{RELISOP}
\Big( f(R-r) \Big)^{\frac{n-1}{n}} \le C\,\Per_r(E,B_{R-r}).\end{equation}
Furthermore,
\begin{eqnarray*} 
\partial (E\setminus B_R)\subseteq
\big(
(\partial E)\setminus B_R\big)\cup\big(
(\partial B_R)\cap E \big).\end{eqnarray*}
Observe that 
\[(E\oplus B_r)\cap (B_{R+r}\setminus B_{R-r})=
\big(E\cap (B_{R+r}\setminus B_{R-r})\big)\cup 
\big((\partial E\oplus B_r)\cap (B_{R+r}\setminus B_{R-r})\big).\]
Consequently
\begin{eqnarray*} 
\big(\partial (E\setminus B_R)\big)\oplus B_r&\subseteq&\Big(
\big(
(\partial E)\oplus B_r\big)\setminus B_{R-r}\Big)\cup\big( (E\oplus B_r)\cap
(B_{R+r}\setminus B_{R-r})\big) \\
&\subseteq&\Big(
\big(
(\partial E)\oplus B_r\big)\cap (B_{R+r}\setminus B_{R-r})\Big)\cup\big( E\cap
(B_{R+r}\setminus B_{R-r})\big) \end{eqnarray*}
and therefore
\begin{equation}\label{ATTEN}
\begin{split} &
{\mathcal{L}}^n 
\Big(\big( (\partial (E\setminus B_R)) \oplus B_r\big)\cap
B_{R+r}\Big)\\ \le\,&
{\mathcal{L}}^n 
\Big(
\big(
(\partial E)\oplus B_r\big)\cap (B_{R+r}\setminus B_{R-r})\Big)
+{\mathcal{L}}^n\big( E\cap
(B_{R+r}\setminus B_{R-r})\big)\\
=\,& 2r\,\Per_r(E,\,B_{R+r}\setminus B_{R-r})+\big( f(R+r)-f(R-r)\big).
\end{split}\end{equation}
Assume also that~$B_{R+r}\subseteq\Omega$.
Then, the minimality of~$E$ in~$B_{R+r}$
and~\eqref{ATTEN}
give that
\begin{eqnarray*}
0 &\le& 2r\,\Big[ \Per_r (E\setminus B_R, \,B_{R+r}) -
\Per_r (E, \,B_{R+r}) \Big]\\
&=& 
{\mathcal{L}}^n 
\Big(\big( (\partial (E\setminus B_R)) \oplus B_r\big)\cap
B_{R+r}\Big)
-2r\,\Per_r (E, \,B_{R+r})
\\ &=&
\big( f(R+r)-f(R-r)\big)
+ 2r\,\Big[\Per_r(E,\,B_{R+r}\setminus B_{R-r})-
\Per_r (E, \,B_{R+r})\Big]\\&=&
\big( f(R+r)-f(R-r)\big)
-
2r\,\Per_r (E, \,B_{R-r}) .
\end{eqnarray*}
This and~\eqref{RELISOP} give that, if~$B_{R+2r}\subseteq\Omega$,
$R\ge2r$ and~$f(R-r)\le \frac{{\mathcal{L}}^n(B_R) }{2}$,
then
$$ 0\le \big( f(R+r)-f(R-r)\big)-\frac{2r\,\Big( f(R-r) \Big)^{\frac{n-1}{n}}}{C}.
$$
That is, if~$R\ge r$ and~$f(R)\le \frac{{\mathcal{L}}^n(B_R) }{2}$,
\begin{equation} \label{BASE}
f(R+2r) \ge f(R)+\frac{2r}{C}
\,\Big( f(R) \Big)^{\frac{n-1}{n}}
.\end{equation}
Now we define, for $k\in\N$, the sequence~$x_k:= f(R_o+2kr)$,
and we claim that, if~$B_{R_o+2kr}\subseteq\Omega$, then
\begin{equation}\label{INDI}
x_k\ge   \,(\omega_o^{\frac{1}{n}}+ 2 c_{\star } k r)^n
\end{equation}
where
\begin{equation}\label{c1}
c_{\star }:=\frac{1}{C\left(n+\frac{2(n-1)r}{C\omega_o^{\frac{1}{n}}}\right)}
% \frac{1}{2n\,C\,\left(
% 1+\frac{(n-1)r}{Cn\omega_o^{\frac1n}}\right)}
,\end{equation}
being~$\omega_o$ as in~\eqref{GA0} and~$C$ as in~\eqref{BASE}.
The proof of~\eqref{INDI} is by induction.
First of all, from~\eqref{GA0}  we have that
$$ x_0 = f(R_o)=\omega_o 
,$$
and so~\eqref{INDI} holds true
when~$k=0$. Now we suppose that it holds true for~$k-1$,
namely
$$ x_{k-1}\ge \,(\omega_o^{\frac{1}{n}}+ 2 c_{\star } (k-1) r)^n.$$
Thus, from~\eqref{BASE},
\begin{equation}
\label{rfdtwyqv6454398e}
\begin{split}
x_k \,&= f(R_o+2(k-1)r+2r)
\\ &\ge
f(R_o+2(k-1)r)+\frac{2r}{C}
\,\Big( f(R_o+2(k-1)r) \Big)^{\frac{n-1}{n}}
\\ &=
x_{k-1}+\frac{2r}{C}
\,\Big( x_{k-1} \Big)^{\frac{n-1}{n}}\\
&= \Big( x_{k-1} \Big)^{\frac{n-1}{n}}\;
\left( \Big( x_{k-1} \Big)^{\frac{1}{n}}+\frac{2r}{C}\right)\\
&\ge
(\omega_o^{\frac{1}{n}}+ 2 c_{\star } (k-1) r)^{n-1}
\;\left(
\omega_o^{\frac{1}{n}}+ 2 c_{\star } (k-1) r
+\frac{2r}{C}\right)\\
&= (\omega_o^{\frac{1}{n}}+ 2 c_{\star } k r)^{n}
\frac{ (\omega_o^{\frac{1}{n}}+ 2 c_{\star } (k-1) r)^{n-1} }{
(\omega_o^{\frac{1}{n}}+ 2 c_{\star } kr)^{n-1} }\;
\frac{ \omega_o^{\frac{1}{n}}+ 2 c_{\star } (k-1) r
+\frac{2r}{C} }{
\omega_o^{\frac{1}{n}}+ 2 c_{\star } k r}\\
&=
(\omega_o^{\frac{1}{n}}+ 2 c_{\star } k r)^{n}
\;\left(1-\frac{2c_{\star }r}{ \omega_o^{\frac{1}{n}}+ 2 c_{\star } kr }
\right)^{n-1}\;\left(1+
\frac{
\frac{2r}{C}-2c_{\star }r
}{ \omega_o^{\frac{1}{n}}+ 2 c_{\star } kr }\right)
.\end{split}\end{equation} 
Now, by a first order Taylor expansion, we see that,
for any~$\tau\in[0,1]$,
$$(1-\tau)^{n-1}\ge 1-(n-1)\tau$$
and therefore
$$ \left(1-\frac{2c_{\star }r}{ \omega_o^{\frac{1}{n}}+ 2 c_{\star } kr }
\right)^{n-1}\ge 1- \frac{2(n-1)\,c_{\star }r}{
\omega_o^{\frac{1}{n}}+ 2 c_{\star } kr }.$$
As a consequence,
\begin{eqnarray*}
&& \left(1-\frac{2c_{\star }r}{ \omega_o^{\frac{1}{n}}+ 2 c_{\star } kr }
\right)^{n-1}\;\left(1+
\frac{
\frac{2r}{C}-2c_{\star }r
}{ \omega_o^{\frac{1}{n}}+ 2 c_{\star } kr }\right)\\
&\ge&\left(
1- \frac{2(n-1)\,c_{\star }r}{
\omega_o^{\frac{1}{n}}+ 2 c_{\star } kr }\right)\;
\;\left(1+
\frac{
\frac{2r}{C}-2c_{\star }r
}{ \omega_o^{\frac{1}{n}}+ 2 c_{\star } kr }\right)\\
&=& 1+
\frac{
\frac{2r}{C}-2c_{\star }r-2(n-1)\,c_{\star }r
}{ \omega_o^{\frac{1}{n}}+ 2 c_{\star } kr }
-
\frac{2(n-1)\,c_{\star }r}{ \omega_o^{\frac{1}{n}}+ 2 c_{\star } kr }\cdot
\frac{
\frac{2r}{C}-2c_{\star }r-2(n-1)\,c_{\star }r
}{ \omega_o^{\frac{1}{n}}+ 2 c_{\star } kr }
\\&\ge&
1+
\frac{
\frac{2r}{C}-2n\,c_{\star }r
}{ \omega_o^{\frac{1}{n}}+ 2 c_{\star } kr }
-
\frac{2(n-1)\,c_{\star }r}{ \omega_o^{\frac{1}{n}}+ 2 c_{\star } kr }\cdot
\frac{
\frac{2r}{C}
}{ \omega_o^{\frac{1}{n}}}
\\ &=&
1+
\frac{
\frac{2r}{C}-2n\,c_{\star }r-\frac{4(n-1)\,c_{\star }r^2}{C
\omega_o^{\frac1n}}
}{ \omega_o^{\frac{1}{n}}+ 2 c_{\star } kr }=1,
\end{eqnarray*}
thanks to~\eqref{c1}.
This and~\eqref{rfdtwyqv6454398e} give that~$x_k \ge 
(\omega_o^{\frac{1}{n}}+ 2 c_{\star } k r)^{n}$, which completes the inductive proof of~\eqref{INDI}.

{F}rom~\eqref{INDI} and~\eqref{c1}, we obtain~\eqref{DENS:EQ1},
\eqref{DENS:EQ2} and~\eqref{DENS:EQ3}.

Now, we prove~\eqref{DENS:EQ4}. To this end,
we take~$k$ as in~\eqref{DENS:EQ5}
and we observe that
$$ x_0\le\dots\le x_{k-1}\le \overline{C} r^n.$$
Hence, for any~$j\in\{1,\dots,k\}$, 
$$ r\,(x_{j-1})^{-\frac{1}{n}}\ge
\overline{C}^{-\frac{1}{n}},$$
thus we deduce from~\eqref{BASE}
that
\begin{eqnarray*} 
&&x_j=f(R_o+2(j-1)r+2r) \ge f(R_o+2(j-1)r)+\frac{2r}{C}
\,\Big( f(R_o+2(j-1)r) \Big)^{\frac{n-1}{n}}\\&&\qquad
= x_{j-1}+\frac{2r}{C}\,(x_{j-1})^{\frac{n-1}{n}}\ge
x_{j-1} \,\left(1+\frac{1}{2C\,\overline C^{\frac1n}}\right).
\end{eqnarray*}
Iterating, we thus obtain
$$ x_k\ge x_{0} 
\,\left(1+\frac{1}{2C\,\overline C^{\frac1n}}\right)^k,$$
that establishes~\eqref{DENS:EQ4}.
This completes the proof of Theorem~\ref{DENSITY}.
\end{proof}

Now we address the compactness and lack of regularity issues
exemplified in Theorem~\ref{FAIL:COM}:

\begin{proof}[Proof of Theorem~\ref{FAIL:COM}]
We start with some preliminary observations.
First of all, if we denote by~$\{e_1,\dots,e_n\}$
the Euclidean basis of~$\R^n$, it is clear that
\begin{equation}\label{FAC:FAC}
{\mathcal{L}}^n\big( B_{1/8}( e_1/2)\cap (B_1\setminus B_{1/2})\big)>0
{\mbox{ and }}
{\mathcal{L}}^n\big( B_{1/8}( e_1)\cap (B_1\setminus B_{1/2})\big)>0.
\end{equation}
Moreover, there exists a constant~$c_\star>0$,
only depending on~$n$, such that, for any~$x\in\overline{B_{3/2}}$
it holds that
\begin{equation}\label{GEO:bal}
{\mathcal{L}}^n\big( B_1(x)\cap (B_1\setminus B_{1/2})\big)\ge c_\star.
\end{equation}
To prove~\eqref{GEO:bal}, we argue for a contradiction: if not,
there exists a sequence of points~$x_k\in\overline{B_{3/2}}$
such that
\begin{equation}\label{GEO:bal:2}
{\mathcal{L}}^n\big( B_1(x_k)\cap (B_1\setminus B_{1/2})\big)\le\frac1k.
\end{equation}
Up to a subsequence, we may assume that~$x_k\to \bar x$
as~$k\to+\infty$, for some~$\bar x\in\overline{B_{3/2}}$,
and, passing to the limit~\eqref{GEO:bal:2}, we obtain that
\begin{equation}\label{GEO:bal:3}
{\mathcal{L}}^n\big( B_1(\bar x)\cap (B_1\setminus B_{1/2})\big)=0.
\end{equation}
Up to a rotation, we can assume that~$\bar x$ is parallel to~$e_1$,
namely~$\bar x=\lambda e_1$, for some~$\lambda\in\left[0,\frac32\right]$.
We define
$$\lambda_\star:=\left\{ \begin{matrix}
1/2 & {\mbox{ if }}\lambda\in\left[0,\frac34\right]\cr
\cr
1 & {\mbox{ if }}\lambda\in\left(\frac34,\frac32\right].
\end{matrix}\right.$$
Notice that, by~\eqref{FAC:FAC}, we have that
\begin{equation}\label{2:FAC:FAC}
{\mathcal{L}}^n\big( 
B_{1/8}( \lambda_\star\,e_1)\cap (B_1\setminus B_{1/2})\big)>0
.\end{equation}
In addition, we note that~$|\bar x-\lambda_\star e_1|=
|\lambda-\lambda_\star|\le 1/2$.
Consequently, if~$p\in B_{1/8}(\lambda_\star e_1)$, we have that~$
|\bar x-p|\le |\bar x-\lambda_\star e_1|+
|\lambda_\star e_1-p|\le \frac12+\frac18<1$,
which gives that~$B_{1/8}(\lambda_\star e_1)\subseteq B_1(\bar x)$.

Therefore, we conclude that~$
B_{1/8}( \lambda_\star\,e_1)\cap (B_1\setminus B_{1/2})\subseteq
B_{1}( \bar x)\cap (B_1\setminus B_{1/2})$.
{F}rom this and~\eqref{2:FAC:FAC}, we obtain that~${\mathcal{L}}^n\big( 
B_{1}( \bar x)\cap (B_1\setminus B_{1/2})\big)>0$,
and this is in contradiction with~\eqref{GEO:bal:3}.
The proof of~\eqref{GEO:bal} is thus completed.

We also notice that, by scaling~\eqref{GEO:bal}, it holds that,
for any~$x\in\overline{B_{3r/2}}$,
\begin{equation}\label{GEO:bal:BIS}
{\mathcal{L}}^n\big( B_r(x)\cap (B_r\setminus B_{r/2})\big)\ge c_\star\,r^n.
\end{equation}
Now we claim that there exists~$\delta_\star>0$, 
only depending on~$n$, such that
\begin{equation}\label{CIC}
\begin{split}
& {\mbox{if $H\subseteq B_r$ and }}
{\mathcal{L}}^n\big( H\cap (B_r\setminus B_{r/2})\big)\ge (1-\delta_\star)\,
{\mathcal{L}}^n(B_1)\,
\left(1-\frac{1}{2^n}\right)\,r^n,\\
&{\mbox{then }}\{0\}\cup\big((\partial H)\oplus B_r \big)\supseteq
\big(\partial(H\cup B_{r/2})\big)\oplus B_r.
\end{split}
\end{equation}
To prove this, let
\begin{equation}\label{quwyd9821wuiuqwgd}
x\in \big(\partial(H\cup B_{r/2})\big)\oplus B_r.\end{equation}
Our aim is to show that
\begin{equation}\label{CIC2}
{\mbox{either $x=0$ or }}
B_r(x)\cap (\partial H)\ne\varnothing,
\end{equation}
since this would imply that~$x\in\{0\}\cup\big((\partial H)\oplus B_r\big)$,
thus establishing~\eqref{CIC}.

Also, since~\eqref{CIC2} is obvious when~$x=0$, we can assume that
\begin{equation}\label{NONZ}
x\ne0.
\end{equation}
Notice that, from~\eqref{quwyd9821wuiuqwgd}, we know that there exists~$y\in B_r(x)\cap \big(\partial(H\cup B_{r/2})\big)$.
Consequently, we can find~$\xi_k\in (H\cup B_{r/2})$
and~$\eta_k\in (\R^n\setminus H)\cap (\R^n\setminus B_{r/2})$
with the property that~$\xi_k\to y$ and~$\eta_k\to y$ as~$k\to+\infty$.

We observe that~$\eta_k\in\R^n\setminus H$: hence, if~$\xi_k\in H$,
it follows that~$y\in\partial H$ and so~\eqref{CIC2} holds true.
Therefore, we can restrict ourselves to the case in which~$\xi_k\in
(B_{r/2}\setminus H)$.
In particular
$$ \frac{r}2\le\lim_{k\to+\infty} |\eta_k|=|y|=\lim_{k\to+\infty} |\xi_k|\le \frac{r}2,$$
and so~$y\in\partial B_{r/2}$.

Consequently, we see that~$|x|\le |y|+|x-y|\le \frac{r}{2}+r=\frac{3r}2$,
and so we are in the position of exploiting~\eqref{GEO:bal:BIS}.
Accordingly, we have that
\begin{equation} \label{CVALXA}
{\mathcal{L}}^n\big( B_r(x)\cap (B_r\setminus B_{r/2})\big)\ge c_\star\,r^n.\end{equation}
In addition, from the hypothesis of~\eqref{CIC}, we find that
\begin{eqnarray*}
{\mathcal{L}}^n(B_1)\,
\left(1-\frac{1}{2^n}\right)\,r^n&=&
{\mathcal{L}}^n (B_r\setminus B_{r/2})\\
&=&
{\mathcal{L}}^n\big( (B_r\setminus B_{r/2})\cap H\big)
+{\mathcal{L}}^n\big( (B_r\setminus B_{r/2})\setminus H\big)\\&
\ge& (1-\delta_\star)\,
{\mathcal{L}}^n(B_1)\,
\left(1-\frac{1}{2^n}\right)\,r^n
+{\mathcal{L}}^n\big( (B_r\setminus B_{r/2})\setminus H\big).
\end{eqnarray*}
This says that
$$ {\mathcal{L}}^n\big( (B_r\setminus B_{r/2})\setminus H\big)\le
\delta_\star\,
{\mathcal{L}}^n(B_1)\,
\left(1-\frac{1}{2^n}\right)\,r^n\le \frac{c_\star}2\,r^n,$$
as long as we choose~$\delta_\star$ appropriately small.
Thus, recalling~\eqref{CVALXA}, we find that
\begin{eqnarray*} 
c_\star\,r^n&\le&
{\mathcal{L}}^n\big( B_r(x)\cap (B_r\setminus B_{r/2})\big)\\
&\le&{\mathcal{L}}^n\big( B_r(x)\cap (B_r\setminus B_{r/2})\cap H\big)
+{\mathcal{L}}^n\Big( \big( B_r(x)\cap (B_r\setminus B_{r/2})\big)\setminus H\Big)\\
&\le&
{\mathcal{L}}^n\big( B_r(x)\cap (B_r\setminus B_{r/2})\cap H\big)
+{\mathcal{L}}^n\big( (B_r\setminus B_{r/2})\setminus H\big)\\
&\le&
{\mathcal{L}}^n\big( B_r(x)\cap (B_r\setminus B_{r/2})\cap H\big)+
\frac{c_\star}2\,r^n,\end{eqnarray*}
which gives that~$ {\mathcal{L}}^n\big( B_r(x)\cap (B_r\setminus B_{r/2})\cap H\big)\ge
\frac{c_\star}2\,r^n$.
In particular, we have that
\begin{equation}\label{NONV}
B_r(x)\cap H\ne\varnothing.\end{equation}
So, we claim that
\begin{equation}\label{NONV2}
B_r(x)\cap (\partial H)\ne\varnothing.\end{equation}
To prove~\eqref{NONV2}, we
suppose the contrary, namely that~$B_r(x)\cap (\partial H)=\varnothing$.
Then, from~\eqref{NONV}
we have that~$B_r(x)\subseteq H$.
In particular, recalling~\eqref{NONZ},
we have that, if~$p_j:=x+\left(r-\frac{1}j\right)\frac{x}{|x|}$, then
$$ |p_j-x|=\left| r-\frac{1}{j}\right|=r-\frac1j<r,$$
for large~$j$. Accordingly, we obtain that~$
p_j\in B_r(x)\subseteq H\subseteq B_r$,
where one assumption in~\eqref{CIC} has been used
for the latter inclusion. 

So, we have found that
$$ r\ge \lim_{j\to+\infty}|p_j|=\lim_{j\to+\infty}
\left|x+\left(r-\frac{1}j\right)\frac{x}{|x|}\right|
=\lim_{j\to+\infty}\left| |x|+\left(r-\frac{1}j\right)\right|=|x|+r.$$
This is a contradiction with~\eqref{NONZ}, and so we have
proved~\eqref{NONV2}.

Then, since~\eqref{NONV2} implies~\eqref{CIC2},
we have thus completed the proof of~\eqref{CIC}.
\bigskip

Now, we deal with the core of the proof of Theorem~\ref{FAIL:COM}.
For this, we observe that
\begin{equation}\label{uip:011}
\begin{split}
{\mathcal{F}}_K(B_r\setminus B_{r/2})\,&:=
\Per_r(B_r\setminus B_{r/2})
-K\,{\mathcal{L}}^n(B_r\setminus B_{r/2})\\
&=\frac{ {\mathcal{L}}^n\big( (\partial(B_r\setminus B_{r/2}))
\oplus B_r\big)}{2r}
- K\,{\mathcal{L}}^n(B_r\setminus B_{r/2})\\
&= 
\frac{ {\mathcal{L}}^n(B_{2r})}{2r}
- K\,{\mathcal{L}}^n(B_r\setminus B_{r/2})\\
&= 2^{n-1}\,{\mathcal{L}}^n(B_1) \,r^{n-1}
-  {\mathcal{L}}^n(B_1)\,\left(1-\frac{1}{2^n}\right)\,K r^n
.\end{split}
\end{equation}
Now we claim that
\begin{equation}\label{77:TC67H11E}
{\mathcal{F}}_K(B_r\setminus B_{r/2})\le {\mathcal{F}}_K(E)\end{equation}
for any bounded set~$E\subseteq\R^n$. To this end,
we distinguish two cases, namely
\begin{eqnarray}
\label{CC:caso1} {\mbox{either}} && 
{\mathcal{L}}^n\big(E\cap (B_r\setminus B_{r/2})\big)\le
(1-\delta_\star)\,
{\mathcal{L}}^n(B_1)\,
\left(1-\frac{1}{2^n}\right)\,r^n\\
{\mbox{ or }}&&\label{CC:caso2}
{\mathcal{L}}^n\big(E\cap (B_r\setminus B_{r/2})\big)>(1-\delta_\star)\,
{\mathcal{L}}^n(B_1)\,
\left(1-\frac{1}{2^n}\right)\,r^n,\end{eqnarray} 
being~$\delta_\star$ the constant in~\eqref{CIC}.

When~\eqref{CC:caso1} holds true, we have that
$$ -{\mathcal{F}}_K(E)\le 
K\,{\mathcal{L}}^n\big( E\cap(B_r\setminus B_{r/2}
)\big) \le 
(1-\delta_\star)\, {\mathcal{L}}^n(B_1)\,\left(1-\frac{1}{2^n}\right)\,K\,r^n.$$
Accordingly, from~\eqref{uip:011}, we have that
\begin{eqnarray*}
&& {\mathcal{F}}_K(B_r\setminus B_{r/2})-{\mathcal{F}}_K(E)\\
&\le& 2^{n-1}\,{\mathcal{L}}^n(B_1) \,r^{n-1}
-  {\mathcal{L}}^n(B_1)\,\left(1-\frac{1}{2^n}\right)\,K r^n
+(1-\delta_\star)\,
{\mathcal{L}}^n(B_1)\,\,\left(1-\frac{1}{2^n}\right)\,K\,r^n\\
&=&
2^{n-1}\,{\mathcal{L}}^n(B_1) \,r^{n-1}
-  
\delta_\star\,{\mathcal{L}}^n(B_1)\,\left(1-\frac{1}{2^n}\right)\,K\,r^n
\\ &\le&0,\end{eqnarray*}
provided that~$K$ is large enough, as prescribed by~\eqref{K LARGE}.
This proves~\eqref{77:TC67H11E} when~\eqref{CC:caso1} holds true, hence we can now
focus on the case in which~\eqref{CC:caso2} is satisfied.

Thanks to~\eqref{CC:caso2}, we can exploit~\eqref{CIC} with
\begin{equation}\label{DEH} H:=E\cap B_r.\end{equation}
In this way, setting
\begin{equation}\label{DEH2} G:= H\cup B_{r/2},\end{equation} we have that
$ \{0\}\cup\big((\partial H)\oplus B_r \big)\supseteq
(\partial G)\oplus B_r$.
In particular, we have that
\begin{equation} \label{102weuhfqowidjblP}
\Per_r(H)\ge \Per_r(G).\end{equation}
We also point out that
$ {\mathcal{L}}^n\big( G\cap (B_{r}\setminus B_{r/2})\big)
={\mathcal{L}}^n\big( H\cap (B_{r}\setminus B_{r/2})\big)$,
thanks to~\eqref{DEH2}. Hence, exploiting~\eqref{102weuhfqowidjblP}, we find that
\begin{equation}\label{della F}
{\mathcal{F}}_K(H)\ge{\mathcal{F}}_K(G).
\end{equation}
In addition, we claim that
\begin{equation} 
\label{BAL}
\Per_r(H)\le \Per_r(E).\end{equation}
To check this, we recall (see formulas~(2.4)-(2.5) in~\cite{MR3023439}) that
\begin{equation} \label{4thsdoio3r}
\Per_r(E\cap B_r)+\Per_r(E\cup B_r)\le\Per_r(E)+
\Per_r(B_r).\end{equation}
Let now~$R\ge r$ be such that~${\mathcal{L}}^n(E\cup B_r)=
{\mathcal{L}}^n(B_R)$. Then, from
the isoperimetric inequality in~\eqref{FORMULA ISP}, we see that
$$ \Per_r(E\cup B_r)\ge\Per_r(B_R)
=\frac{{\mathcal{L}}^n(B_{R+r})}{2r}\ge
\frac{{\mathcal{L}}^n(B_{2r})}{2r}=\Per_r(B_r)
.$$
Hence, we insert this inequality into~\eqref{4thsdoio3r}
and we obtain~\eqref{BAL}, as desired.

We also notice that, by~\eqref{DEH}, we have that~$ {\mathcal{L}}^n\big(H\cap(B_r\setminus B_{r/2})\big)=
{\mathcal{L}}^n\big(E\cap(B_r\setminus B_{r/2})\big)$,
and so
\begin{equation}\label{uip:012}
{\mathcal{F}}_K(H)\le
{\mathcal{F}}_K(E),
\end{equation}
thanks to~\eqref{BAL}.

Let also~$\rho\ge0$ be such that
\begin{equation}\label{VOL:G}
{\mathcal{L}}^n(G)=
{\mathcal{L}}^n(B_\rho).
\end{equation}
We point out that, by~\eqref{DEH} and~\eqref{DEH2},
\begin{equation}\label{356}
B_{r/2}\subseteq G\subseteq B_r,
\end{equation}
and so
\begin{equation}\label{JK:89:00:02}
\rho\in \left[ \frac{r}2,r\right].\end{equation}
Also, making use of the isoperimetric inequality in~\eqref{FORMULA ISP},
we see that
\begin{equation}\label{ST:GSOTTL:1}
\Per_r(G)\ge \Per_r(B_\rho)=\frac{
{\mathcal{L}}^n(B_{r+\rho})
}{2r}=\frac{{\mathcal{L}}^n (B_1)\,(r+\rho)^n}{2r}.
\end{equation}
Furthermore,
\begin{equation}\label{ST:GSOTTL:2}\begin{split}&
{\mathcal{L}}^n \big( G\cap(B_r\setminus B_{r/2})\big)=
{\mathcal{L}}^n ( G\cap B_r)-
{\mathcal{L}}^n (G\cap B_{r/2})
= {\mathcal{L}}^n (G)-{\mathcal{L}}^n (B_{r/2})
\\ &\quad\quad=
{\mathcal{L}}^n (B_\rho)-{\mathcal{L}}^n (B_{r/2})=
{\mathcal{L}}^n (B_1)\,\left(
\rho^n-\left(\frac{r}{2}\right)^n\right),\end{split}\end{equation}
thanks to~\eqref{VOL:G} and~\eqref{356}.

Hence, by~\eqref{ST:GSOTTL:1} and~\eqref{ST:GSOTTL:2},
we have that
\begin{equation}\label{fc0019e090}
{\mathcal{F}}_K(G) \ge 
\frac{{\mathcal{L}}^n (B_1)\,(r+\rho)^n}{2r}
-{\mathcal{L}}^n (B_1)\,K\,\left(
\rho^n-\left(\frac{r}{2}\right)^n\right) =:\Phi(\rho).
\end{equation}
We notice that, for any~$t\in \left[ \frac{r}2,\,r\right]$, 
\begin{eqnarray*} \Phi'(t)&=& n\,{\mathcal{L}}^n (B_1)\,\left(
\frac{(r+t)^{n-1}}{2r}
-K\, t^{n-1}
\right)\\&=&
n\,{\mathcal{L}}^n (B_1)\,t^{n-1}\left(
\frac{1}{2r} \,\left(\frac{r}{t}+1\right)^{n-1}
-K
\right)\\ &\le&
n\,{\mathcal{L}}^n (B_1)\,t^{n-1}\left(
\frac{1}{2r} \,\left(\frac{r}{r/2}+1\right)^{n-1}
- K
\right)\\&\le&0,
\end{eqnarray*}
as long as~$K$ is large enough, as prescribed in~\eqref{K LARGE}.
Therefore, recalling~\eqref{uip:011}
and~\eqref{JK:89:00:02}, we have that
$$ {\mathcal{F}}_K(B_r\setminus B_{r/2})=
{\mathcal{L}}^n (B_1)\,\left[ 2^{n-1} r^{n-1}
-K\,\left(
1-\frac{1}{2^n}\right)\,r^n\right]
=
\Phi(r)=\inf_{t\in\left[ \frac{r}2,\,r\right]}\Phi(t)\le
\Phi(\rho).$$
Hence, we insert this information into~\eqref{fc0019e090},
and we conclude that~${\mathcal{F}}_K(G)\ge {\mathcal{F}}_K
(B_r\setminus B_{r/2})$.
{F}rom this, \eqref{uip:012} and~\eqref{della F}, we conclude that
$$ {\mathcal{F}}_K(E)\ge {\mathcal{F}}_K(H)\ge {\mathcal{F}}_K(G)
\ge {\mathcal{F}}_K(B_r\setminus B_{r/2}),$$
which completes the proof of~\eqref{77:TC67H11E}.

Now, for any (arbitrarily ugly) set~$U\subseteq B_{r/2}$,
we set~$E_U:= (B_r\setminus B_{r/2})\cup U$. We notice
that~$ (\partial E_U)\oplus B_r = B_{2r}=
(B_r\setminus B_{r/2})\oplus B_r$,
and also~$ {\mathcal{L}}^n\big(E_U\cap (B_r\setminus B_{r/2})\big)={\mathcal{L}}^n
(B_r\setminus B_{r/2})$,
and therefore
$$ {\mathcal{F}}_K(E_U)=
{\mathcal{F}}_K(B_r\setminus B_{r/2}).$$
Hence, from~\eqref{77:TC67H11E}, we have that~$E_U$ is
also a minimizer for~${\mathcal{F}}_K$, from which the claims
in Theorem~\ref{FAIL:COM} plainly follow.
\end{proof}

%%%%%%%%%%%%%%%%%%%%%%%%%%%%%%%%%%%%%%%%%%%%%%%%%%%%%%%%%%%%%%%%%%%%%%%%%%%%%%%%%%%%%%
%%%%%%%%%%%%%%%%%%%%%%%%%%%%%%%%%%%%%%SEZIONE PLANELIKE %%%%%%%%%%%%%%%%%%%%%%%%%%%%%%%%%%%%%%}%]
%%%%%%%%%%%%%%%%%%%%%%%%%%%%%%%%%%%%%%%%%%%%%%%%%%%%%%%%%%%%%%%%%%%%%%%%%%%%%%%%%%%%

\section{Planelike minimizers in periodic media -- Proof of
Theorem~\ref{PLANELIKE:TH}} \label{SA:1023}

In this section we establish the existence 
of planelike minimizers for periodic
volume perturbations of~$\Per_r$.

\begin{proof}[Proof of Theorem~\ref{PLANELIKE:TH}]
The proof is given in two steps: in the first one, we fix a rational slope $\omega$ and we provide the construction of 
a planelike minimizer $E^*_\omega$ which is also $\omega$- periodic.  Then, in the second step, we consider irrational slopes by means of an approximation procedure.
\smallskip

{\bf Step 1: construction of planelike minimizers with rational slope}.
The idea of the proof is to perform an argument based on a constrained
minimal minimizer procedure, as in~\cite{MR1852978}.
A major difference with~\cite{MR1852978} here is that
optimal density estimates at small scales do not hold,
hence the width of the strip may depend, in principle,
on~$r$. Indeed, roughly speaking, here one needs an initial density
to improve the density in the large, and so, to let the
density reach a uniform threshold, a large number (in dependence of~$r$)
of fundamental cubes
may be needed, and
this has a rather
strong consequence on the energy estimates when~$r$ is small.

Hence, the proof of this step will be performed in two parts: 
first, we obtain
an initial bound on the width of the strip that depends on~$r$,
and then we improve this bound up to a uniform scale.
This method will combine
the minimal minimizer argument in~\cite{MR1852978} with
an ad-hoc procedure of finely selecting appropriating cubes
and performing a cut at a suitable level.
These estimates will be based on a fine analysis of cubes,
to detect local densities and energy contributions.

The details of the proof go like this.
We consider a ``fundamental domain'' for the~$\omega$-periodicity,
i.e. we take~$K_1,\dots,K_{n-1}\in\Z^n$ which are linearly independent
and such that~$\omega\cdot K_j=0$ for any~$j\in\{1,\dots,n-1\}$, and we set
$$ F_\omega := \big\{ t_1 K_1+\dots+t_{n-1}K_{n-1},\quad\;
t_1,\dots,t_{n-1}\in (0,1)\big\}.$$
Notice that the existence of~$K_1,\dots,K_{n-1}$
is a consequence of the rationality of~$\omega$ in~\eqref{RAT}.

Given~$M\ge2$, we also consider the parallelepipedon
\begin{eqnarray*} S_{\omega,M}&:=& \big\{ t_1 K_1+\dots+t_{n-1}K_{n-1}+t_n\omega,\quad\;
t_1,\dots,t_{n-1}\in (0,1), \quad t_n\in(-M,M)\big\}\\
&=& \big\{ p+t_n\omega,\quad\;p\in F_\omega
, \quad t_n\in(-M,M)\big\}.
\end{eqnarray*}
We consider the functional
$$ {\mathcal{F}}_{\omega,M}(E):=
\Per_r(E, S_{\omega,2M})+\int_{E\cap S_{\omega,2M}} g(x)\,dx.$$
We now introduce the set of periodic constrained minimizers
for this functional. Namely we define~${\mathcal{C}}_{\omega,M}$
the family of sets~$E\subseteq\R^n$ which are~$\omega$-periodic and such
that
$$
\{\omega\cdot x\le -M\}\subseteq E\subseteq
\{\omega\cdot x\le M\}. $$
Let also~$L_\omega:=\{\omega\cdot x\le0\}$.
Then
\begin{equation}\label{6rtd76tr3gyuaj}
\Per_r(L_\omega, S_{\omega,2M})\le C {\mathcal{H}}^{n-1}(F_\omega),\end{equation}
for some~$C>0$.

We also consider the family of finite overlapping dilated cubes
$$
{\mathcal{Q}}:=\{ j+[0,n]^n,\quad j\in\Z^n\}.$$
We define~${\mathcal{Q}}_M$
the family of
cubes~$Q\in {\mathcal{Q}}$ which intersect~$\{\omega\cdot x=\pm M\}$.
The fact that~$g$ has zero average in each~$Q\in{\mathcal{Q}}$ implies that
$$ \left| \int_{E\cap S_{\omega,2M}} g(x)\,dx\right|\le
\sum_{Q\in {\mathcal{Q}}_{2M}} \int_{Q} |g(x)|\,dx\le
\|g\|_{L^\infty(\R^n)}\,\sum_{Q\in {\mathcal{Q}}_{2M}}{\mathcal{L}}^n(Q)
\le C\eta\, {\mathcal{H}}^{n-1}(F_\omega),
$$
up to renaming~$C>0$, and therefore,
in view of~\eqref{6rtd76tr3gyuaj}, it holds that
\begin{equation}\label{PIANO MI}
{\mathcal{F}}_{\omega,M}(L_\omega)\le C {\mathcal{H}}^{n-1}(F_\omega).
\end{equation}
This says that there exists at least one set in~${\mathcal{C}}_{\omega,M}$
with finite energy, hence we can proceed to the minimization of the functional.
The existence of the minimum in this case follows along the lines
of Theorem~\ref{DEBP}.

So we define~${\mathcal{M}}_{\omega,M}$
as the family of sets~$E\in{\mathcal{C}}_{\omega,M}$ such that
$$ {\mathcal{F}}_{\omega,M}(E)=\inf_{ F\in {\mathcal{C}}_{\omega,M}}
{\mathcal{F}}_{\omega,M}(F).$$
Following a classical idea of~\cite{MR1852978}, we now define
the minimal minimizer as
$$ E^*_{\omega,M}:=\bigcap_{E\in {\mathcal{M}}_{\omega,M}} E.$$
We remark
%\footnote{Formulas similar to~\eqref{CAPCUP}
%are already present in the literature,
%see e.g. Section~2.1 in~\cite{MR3023439}
%and Proposition~3.2 in~\cite{MR2655948}.
%For completeness, we remark that a self-contained proof
%can be obtained 
%by observing that, for any function~$f$ and~$g$,
%and for any~$p$, $q\in B_r(x)$, we have that
%\begin{eqnarray*}
%&&\big( \min\{f,g\}(p)- \min\{f,g\}(q)\big)+\big( \max\{f,g\}(p)- \max\{f,g\}(q)\big)\\&&\qquad
%=\big( \min\{f,g\}(p)+ \max\{f,g\}(p)\big)- \big(\min\{f,g\}(q)+ \max\{f,g\}(q)\big)\\
%&&\qquad =\big( f(p)+g(p)\big)-
%\big( f(q)+g(q)\big)=
%\big( f(p)-f(q)\big)+\big(g(p)+g(q)\big)\le
%\osc_{B_r(x)} f+\osc_{B_r(x)}g.\end{eqnarray*}
%Therefore
%$$ \osc_{B_r(x)} \min\{f,g\}+\osc_{B_r(x)}\max\{f,g\}\le
%\osc_{B_r(x)} f+\osc_{B_r(x)}g,$$
%and thus
%\begin{equation*} 
%\int_\Omega 
%\osc_{B_r(x)} \min\{f,g\}\,dx+
%\int_\Omega\osc_{B_r(x)}\max\{f,g\}\,dx
%\le \int_\Omega\osc_{B_r(x)} f\,dx+
%\int_\Omega\osc_{B_r(x)} g\,dx.\end{equation*}
%So, taking~$f:=\chi_E$ and~$g:=\chi_F$,
%observing that~$\min\{\chi_E,\chi_F\}=\chi_{E\cap F}$
%and~$\max\{\chi_E,\chi_F\}=\chi_{E\cup F}$,
%and exploiting~\eqref{COAREA}, one obtains~\eqref{CAPCUP}.}
that
\begin{equation}\label{CAPCUP} \Per_r(E\cap F,\Omega)+
\Per_r(E\cup F,\Omega)\le \Per_r(E,\Omega)+\Per_r(F,\Omega),\end{equation}
for any~$E$, $F\subseteq\R^n$ and any domain~$\Omega$,
and thus
\begin{equation}\label{Lia18}
{\mathcal{F}}_{\omega,M}(E\cap F)+
{\mathcal{F}}_{\omega,M}(E\cup F)\le
{\mathcal{F}}_{\omega,M}(E)+
{\mathcal{F}}_{\omega,M}(F).
\end{equation}
By~\eqref{Lia18}, we have that~$E^*_{\omega,M}\in
{\mathcal{M}}_{\omega,M}$, that is the minimal minimizer is indeed a minimizer.
Moreover, $E^*_{\omega,M}$ satisfies the inclusion properties (for a proof of this we refer to \cite[Lemma 6.5]{MR1852978})
\begin{equation}\label{BIRL}
\begin{split}&
{\mbox{if $k\in\Z^n$ and~$\omega\cdot k\le 0$, then~$E^*_{\omega,M}+k
\subseteq E^*_{\omega,M}$;}}\\
&
{\mbox{if $k\in\Z^n$ and~$\omega\cdot k\ge 0$, then~$E^*_{\omega,M}+k
\supseteq E^*_{\omega,M}$.}}
\end{split}
\end{equation}
Consequently, since~$E^*_{\omega,M}$ is the smallest possible minimizers,
\begin{equation*}
\begin{split}&
{\mbox{if $B_{n}(p)\cap E^*_{\omega,M}=\varnothing$,
then~$E^*_{\omega,M}\subseteq \{ \omega\cdot(p-x)\le n\}$}}\\
{\mbox{and }}\quad
&
{\mbox{if $B_{n}(p)\subseteq E^*_{\omega,M}$,
then~$E^*_{\omega,M}\supseteq \{ \omega\cdot(p-x)\ge -n\}$.}}
\end{split}
\end{equation*}
We now divide the cubes in~${\mathcal{Q}}$ according to their
``color'', i.e. their density properties with respect to the set~$E^*_{\omega,M}$
(pictorially, we think that the set~$E^*_{\omega,M}$ is ``black''
and its complement is ``white'').

Namely, we consider the ``family of black cubes'' given by
$$ {\mathcal{Q}}_{{\rm{Bl}}} := \{ Q\in{\mathcal{Q}} {\mbox{ s.t. }}Q\subseteq E^*_{\omega,M}
\}$$
and
the ``family of white cubes''
$$ {\mathcal{Q}}_{\rm{{Wh}}} := \{ Q\in{\mathcal{Q}} {\mbox{ s.t. }}
Q\cap E^*_{\omega,M}=\varnothing
\}.$$
We also take into account the ``family of grey cubes''
\begin{eqnarray*}
{\mathcal{Q}}_{\rm{{Gr}}}&:=&
{\mathcal{Q}}\setminus\big( {\mathcal{Q}}_{\rm{{Bl}}}\cup
{\mathcal{Q}}_{\rm{{Wh}}}\big)\\
&=&
 \{ Q\in{\mathcal{Q}} {\mbox{ s.t. }}
 Q\setminus E^*_{\omega,M}\ne\varnothing
{\mbox{ and }}
Q\cap E^*_{\omega,M}\ne\varnothing
\}.
\end{eqnarray*}
We also subdivide the grey cubes into the ones which
are ``foggy black'' and the ones which are ``foggy white'':
the first family contains cubes with a sufficient density of~$E^*_{\omega,M}$,
while the second family contains cubes with a sufficient density of
the complement of~$E^*_{\omega,M}$, being the notion of
``sufficient density'' the one compatible with uniform scales
in the density estimates of Theorem~\ref{DENSITY}.
That is, we define
\begin{eqnarray*}
&&{\mathcal{Q}}_{\rm{{f.Bl}}} 
:=
\{ Q\in{\mathcal{Q}}_{\rm{{Gr}}} {\mbox{ s.t. }}
{\mathcal{L}}^n( Q\cap E^*_{\omega,M})\ge r^n
\}\\{\mbox{and }}&&
{\mathcal{Q}}_{\rm{{f.Wh}}}:=
\{ Q\in{\mathcal{Q}}_{\rm{{Gr}}} {\mbox{ s.t. }}
{\mathcal{L}}^n( Q\setminus E^*_{\omega,M})\ge r^n
\}.\end{eqnarray*}
Notice that, since~$r\in(0,1)$,
$$ {\mathcal{Q}}_{\rm{{Gr}}}=
{\mathcal{Q}}_{\rm{{f.Bl}}}\cup{\mathcal{Q}}_{\rm{{f.Wh}}}.$$
On the other hand, in general, we have that~${\mathcal{Q}}_{\rm{{f.Bl}}} \cap
{\mathcal{Q}}_{\rm{{f.Wh}}}\ne\varnothing$, since there might be cubes
with sufficiently high density of both~$E^*_{\omega,M}$ and its complement
(these cubes are, in some sense, ``multicolored'' inside). So, we define
\begin{eqnarray*}
{\mathcal{Q}}_{\rm{{Mu}}}&:=&
\{ Q\in {\mathcal{Q}}_{\rm{{f.Bl}}}\cap{\mathcal{Q}}_{\rm{{f.Wh}}}\}
\\&=&
\Big\{ Q\in{\mathcal{Q}}_{\rm{{Gr}}} {\mbox{ s.t. }} \min\big\{
{\mathcal{L}}^n( Q\cap E^*_{\omega,M}),\,
{\mathcal{L}}^n( Q\setminus E^*_{\omega,M})\big\}\ge r^n
\Big\}.
\end{eqnarray*}
Notice that the cubes in~${\mathcal{Q}}_{\rm{{f.Bl}}}\setminus {\mathcal{Q}}_{\rm{{Mu}}}$
have a sufficiently high density of~$E^*_{\omega,M}$ and a rather low
density of its complement, so they ``look almost black''. For this reason,
we set
\begin{eqnarray*}
{\mathcal{Q}}_{\rm{{a.Bl}}}&:=&
\{ Q\in {\mathcal{Q}}_{\rm{{f.Bl}}}\setminus{\mathcal{Q}}_{\rm{{Mu}}}\}
\\&=&
\big\{ Q\in{\mathcal{Q}}_{\rm{{Gr}}} {\mbox{ s.t. }} 
{\mathcal{L}}^n( Q\cap E^*_{\omega,M})\ge r^n>
{\mathcal{L}}^n( Q\setminus E^*_{\omega,M})
\big\}.
\end{eqnarray*}
Similarly, we define the family of almost white cubes as
\begin{eqnarray*}
{\mathcal{Q}}_{\rm{{a.Wh}}}&:=&
\{ Q\in {\mathcal{Q}}_{\rm{{f.Wh}}}\setminus{\mathcal{Q}}_{\rm{{Mu}}}\}
\\&=&
\big\{ Q\in{\mathcal{Q}}_{\rm{{Gr}}} {\mbox{ s.t. }} 
{\mathcal{L}}^n( Q\setminus E^*_{\omega,M})\ge r^n>
{\mathcal{L}}^n( Q\cap E^*_{\omega,M})
\big\}.
\end{eqnarray*}
We are now going to show that the strip is divided
into five ordered ``color layers'': on the bottom stay all the black cubes,
then the almost black ones, then cubes of multicolor type, then almost white cubes
and finally white cubes on the top (rigorous statements below).
We also estimate carefully the width of these layers.

To this end, we observe that the ``color density'' of the cubes
is monotone with respect to~$\omega$, in the sense that
the color of an upper translation is more pale than the color of a lower translation.
The precise statement goes as follows: we claim that, for any~$k\in\Z^n$
with~$\omega\cdot k\ge0$, we have that
\begin{equation}\label{COLO}
{\mathcal{L}}^n \big((Q+k)\cap E^*_{\omega,M}\big)\le
{\mathcal{L}}^n \big(Q\cap E^*_{\omega,M}\big)\le
{\mathcal{L}}^n \big((Q-k)\cap E^*_{\omega,M}\big).
\end{equation}
To check this, we exploit~\eqref{BIRL} to see that~$
E^*_{\omega,M}-k\subseteq E^*_{\omega,M}\subseteq E^*_{\omega,M}+k$ and therefore
\begin{eqnarray*}&& \big( (Q+k)\cap E^*_{\omega,M}\big)-k=
Q\cap (E^*_{\omega,M}-k) \subseteq Q\cap E^*_{\omega,M}
\\{\mbox{and
}}&& \big( (Q-k)\cap E^*_{\omega,M}\big)+k=
Q\cap (E^*_{\omega,M}+k) \supseteq Q\cap E^*_{\omega,M}.
\end{eqnarray*}
Accordingly
\begin{eqnarray*}&& 
{\mathcal{L}}^n
\big( (Q+k)\cap E^*_{\omega,M}\big)={\mathcal{L}}^n
\Big(
\big( (Q+k)\cap E^*_{\omega,M}\big)-k\Big)\le
{\mathcal{L}}^n
\big(
Q\cap E^*_{\omega,M}\big)
\\{\mbox{and
}}&& 
{\mathcal{L}}^n
\big( (Q-k)\cap E^*_{\omega,M}\big) 
=
{\mathcal{L}}^n
\Big(
\big( (Q-k)\cap E^*_{\omega,M}\big)+k\Big)\geq
{\mathcal{L}}^n
\big(Q\cap E^*_{\omega,M}\big),
\end{eqnarray*}
thus proving~\eqref{COLO}.

As a consequence of~\eqref{COLO}, we have that,
for any~$k\in\Z^n$ with~$\omega\cdot k\ge0$,
\begin{equation}\label{LAYERS}
\begin{split}
& {\mbox{if $Q\in {\mathcal{Q}}_{\rm{{Bl}}}$,
then~$Q+k\in {\mathcal{Q}}_{\rm{{Bl}}}\cup
{\mathcal{Q}}_{\rm{{Gr}}}\cup {\mathcal{Q}}_{\rm{{Wh}}}$,}}\\
& {\mbox{if $Q\in {\mathcal{Q}}_{\rm{{Gr}}}$,
then~$Q+k\in 
{\mathcal{Q}}_{\rm{{Gr}}}\cup {\mathcal{Q}}_{\rm{{Wh}}}$,}}\\
& {\mbox{if $Q\in {\mathcal{Q}}_{\rm{{Wh}}}$,
then~$Q+k\in {\mathcal{Q}}_{\rm{{Wh}}}$,}}\\
& {\mbox{if $Q\in {\mathcal{Q}}_{\rm{{a.Bl}}}$,
then~$Q+k\in {\mathcal{Q}}\setminus {\mathcal{Q}}_{\rm{{Bl}}}$,}}\\
& {\mbox{if $Q\in {\mathcal{Q}}_{\rm{{a.Wh}}}$,
then~$Q+k\in {\mathcal{Q}}_{\rm{{a.Wh}}}\cup
{\mathcal{Q}}_{\rm{{Wh}}}$,}}
\end{split}
\end{equation}
and similar statements hold in the case~$\omega\cdot k\le0$.

We now point out that,
if~$Q\in {\mathcal{Q}}_{\rm{{Gr}}}$, then~$Q\cap(\partial 
E^*_{\omega,M})\ne\varnothing$ and therefore
\begin{equation}\label{PER:cube}
\Per_r(E^*_{\omega,M}\cap Q,S_{\omega,2M})\ge
c r^{n-1},\end{equation}
for some~$c>0$.
We also observe that, for any~$E\subseteq\R^n$,
\begin{equation}\label{g int men}
\left| \int_{E\cap Q}g(x)\,dx\right|\le
\|g\|_{L^\infty(\R^n)}\,\min\big\{ {\mathcal{L}}^n(E\cap Q),\,
{\mathcal{L}}^n(Q\setminus E)
 \big\}.\end{equation}
To check this, let us first suppose that~${\mathcal{L}}^n(E\cap Q)\le
{\mathcal{L}}^n(Q\setminus E)$. Then,
$$ \left| \int_{E\cap Q}g(x)\,dx\right|\le
\int_{E\cap Q}|g(x)|\,dx\le \|g\|_{L^\infty(\R^n)}\,{\mathcal{L}}^n(E\cap Q),$$
which gives~\eqref{g int men} in this case.
Conversely, if~${\mathcal{L}}^n(E\cap Q)>
{\mathcal{L}}^n(Q\setminus E)$ we use that~$g$ has zero average and we write
\begin{eqnarray*}&&\left| \int_{E\cap Q}g(x)\,dx\right|=
\left| \int_{ Q}g(x)\,dx-\int_{E\cap Q}g(x)\,dx\right|\\&&\qquad
=\left|
\int_{Q\setminus E}g(x)\,dx\right|\le \|g\|_{L^\infty(\R^n)}\,{\mathcal{L}}^n(Q\setminus E),\end{eqnarray*}
thus completing the proof of~\eqref{g int men}.

In view of~\eqref{PER:cube} and~\eqref{g int men}, we know that
\begin{equation}\label{GREYSTAR}
\begin{split}
&{\mbox{for any~$Q\in{\mathcal{Q}}_{\rm{{Gr}}}$, }}\\&\qquad
{\mathcal{F}}_{\omega,M}(E^*_{\omega,M}\cap Q,
S_{\omega,2M})\ge c r^{n-1}-\|g\|_{L^\infty(\R^n)}\,r^n\ge
r^{n-1} (c-\eta r)\ge\frac{c r^{n-1}}{2},\end{split}
\end{equation}
provided that~$\eta$ is small enough.

On the other hand, if~$Q\in{\mathcal{Q}}_{\rm{{Mu}}}$,
we are in the uniform density
setting of~\eqref{DENS:EQ3} and, consequently,
by~\eqref{DENS:EQ1} we can write that
\begin{equation}\label{unide}
\min\big\{
{\mathcal{L}}^n ( E^*_{\omega,M}\cap Q'),
{\mathcal{L}}^n ( Q'\setminus E^*_{\omega,M}),
\big\}
\ge c
,\end{equation}
up to renaming~$c>0$, where~$Q'$ is the dilation of~$Q$ with respect to
its center by a factor~$2$ (we stress that the condition~$Q\in{\mathcal{Q}}_{\rm{{Mu}}}$
has been used here to guarantee an initial estimate on the
density, which makes the constants in 
Theorem~\ref{DENSITY} uniform).

Then, from~\eqref{unide} and the relative isoperimetric
inequality in Theorem~\ref{ISOPER}, up to renaming~$c>0$, we have that
if~$Q\in{\mathcal{Q}}_{\rm{{Mu}}}$, then
\begin{equation*}
\Per_r(E^*_{\omega,M}\cap Q',S_{\omega,2M})\ge
c .\end{equation*}
Therefore
\begin{equation}\label{FOGGY}
\begin{split}
&{\mbox{for any~$Q\in{\mathcal{Q}}_{\rm{{Mu}}}$, }}\\&\qquad
{\mathcal{F}}_{\omega,M}(E^*_{\omega,M}\cap Q',S_{\omega,2M})\ge 
c-\|g\|_{L^\infty(\R^n)}\,{\mathcal{L}}^n(Q')\ge\frac{c}{2},\end{split}
\end{equation}
provided that~$\eta$ is small enough.

Now we denote by~$J_{\rm{{Mu}}}$, $J_{\rm{{a.Bl}}}$
and~$J_{\rm{{a.Wh}}}$ the number of cubes in~${\mathcal{Q}}_{\rm{{Mu}}}$, 
${\mathcal{Q}}_{\rm{{a.Bl}}}$
and~${\mathcal{Q}}_{\rm{{a.Wh}}}$, respectively.
Then, up to renaming constants, we deduce from~\eqref{GREYSTAR} and~\eqref{FOGGY}
that
\begin{equation*}
\begin{split}
{\mathcal{F}}_{\omega,M}(E^*_{\omega,M},S_{\omega,2M})
\ge cr^{n-1}\,
(J_{\rm{{a.Bl}}}+J_{\rm{{a.Wh}}})
+c\,J_{\rm{{Mu}}}.\end{split}
\end{equation*}
Comparing with~\eqref{PIANO MI} and using minimality, we thus obtain that
$$ r^{n-1}\,
(J_{\rm{{a.Bl}}}+J_{\rm{{a.Wh}}})
+c\,J_{\rm{{Mu}}}
\le C {\mathcal{H}}^{n-1}(F_\omega),$$
up to renaming~$C>0$.
Hence, in view of the layer structure described in~\eqref{LAYERS},
we have that 
\begin{equation}\label{STRISCIA1}
{\mbox{the family of cubes in~${\mathcal{Q}}_{\rm{{Mu}}}$
lies in a strip of width at most~$C$,}}\end{equation}
while 
\begin{equation}\label{STRISCIA2}
{\mbox{the families of cubes in~${\mathcal{Q}}_{\rm{{a.Bl}}}$
and in~${\mathcal{Q}}_{\rm{{a.Wh}}}$
lie in strips of width at most~$\frac{C}{r^{n-1}}$.}}\end{equation}
We observe that the bound in~\eqref{STRISCIA1} is already satisfactory,
but the one in~\eqref{STRISCIA2} needs to be improved if we want to arrive at a strip
of uniform width (independent of~$r$). 
That is, we are now in a situation in which ``almost white'' or ``almost black''
cubes may
have a long tail in the strip when~$r$ is small,
and we want to rule out this possibility.

\begin{figure}[h]
    \centering
    \includegraphics[width=7.5cm]{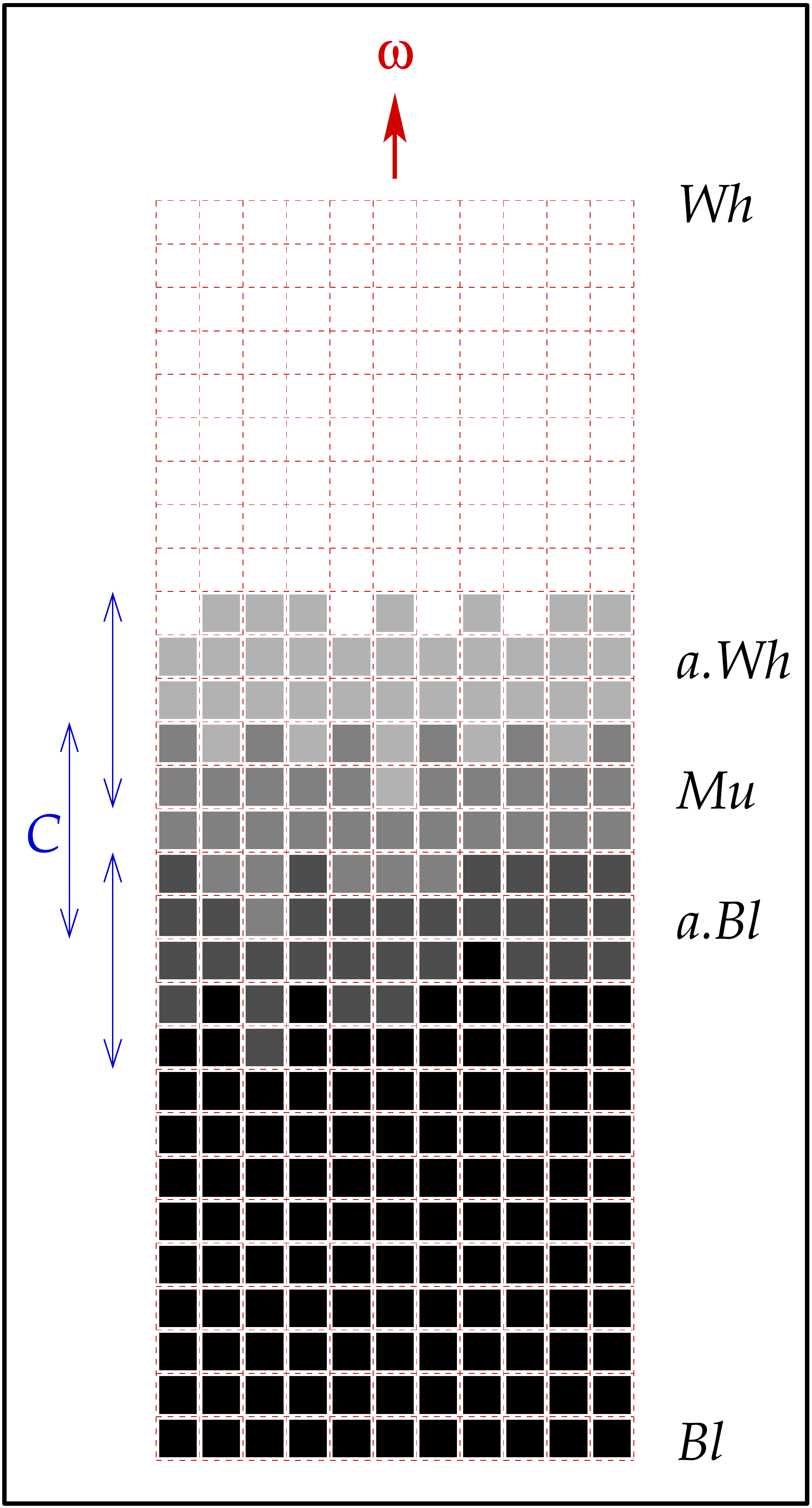}
    \caption{\it {{The geometry of the
colored cubes in~\eqref{STRISCIA1}
and~\eqref{STRISCIA3}.}}}
    \label{CUB}
\end{figure}

For this,
we need a careful procedure of cutting cubes in~${\mathcal{Q}}_{\rm{{a.Wh}}}$.
The idea is that once we have a cube which is ``almost white'' we can color
the region above it in full white, gaining energy. 

The goal is thus to replace~\eqref{STRISCIA2}
with
\begin{equation}\label{STRISCIA3}
{\mbox{the families of cubes in~${\mathcal{Q}}_{\rm{{a.Bl}}}$
and in~${\mathcal{Q}}_{\rm{{a.Wh}}}$
lie in strips of width at most~$C$,}}\end{equation}
up to renaming~$C$ (the situation of formulas~\eqref{STRISCIA1}
and~\eqref{STRISCIA3}
is graphically depicted in Figure~\ref{CUB}).
So, if we can bound the width in~\eqref{STRISCIA2}
with a uniform bound, we are done; otherwise
suppose that, for instance,
${\mathcal{Q}}_{\rm{{a.Wh}}}$ occupies a strip of width~$W_r\ge 2n$,
possibly depending on~$r$ (from~\eqref{STRISCIA2}, we only know that~$W_r\le C/r^{n-1}$),
say~$\{C_o\le \omega\cdot x\le C_o+W_r\}$
(notice that the position of the lower boundary of this strip
is uniformly bounded, thanks to~\eqref{STRISCIA1}, so we denoted it by~$C_o$
for the sake of clarity).

The idea is now to replace~$E^*_{\omega,M}$ with~$E^*_{\omega,M}\cap
\{\omega\cdot x\le C_o+\sqrt{n}\}$. To compute the effect of this
cut, let us consider that, at levels~$\{\omega\cdot x\in[ C_o,C_o+n]\}$,
we may have created additional $r$-perimeter adding portions
of~$\{\omega\cdot x=C_o+\sqrt{n}\}$ to the boundary of the set.
Since this portion is flat,
the cut procedure has produced an energy increasing for the
$r$-perimeter of size at most~$C\,{\mathcal{H}}^{n-1}(F_\omega)\,r^{n-1}$.
As for the bulk energy produced by~$g$,
in each cube~$Q$ in~$\{\omega\cdot x=C_o+\sqrt{n}\}$,
we have produced an energy increasing of at most
$$ \|g\|_{L^\infty(\R^n)}\,\min\big\{ {\mathcal{L}}^n(E^*_{\omega,M}\cap Q),\,
{\mathcal{L}}^n(Q\setminus E^*_{\omega,M}) \big\}
\le \eta\, {\mathcal{L}}^n(E^*_{\omega,M}\cap Q)\le \eta r^n,$$
thanks to~\eqref{g int men} and to the fact that~$Q\in {\mathcal{Q}}_{\rm{{a.Wh}}}$.
That is, the total bulk energy
increased at levels~$\{\omega\cdot x\in[ C_o,C_o+n]\}$ is bounded by~$C\,\eta\,
{\mathcal{H}}^{n-1}(F_\omega)\,r^n$.
Summarizing, the modifications of the
cubes in~${\mathcal{Q}}_{\rm{{a.Wh}}}$
at levels~$\{\omega\cdot x\in[ C_o,C_o+n]\}$
produce an energy increasing bounded by
\begin{equation}\label{jiIKJAOJAHFUYHGAKJHIS}
C\,{\mathcal{H}}^{n-1}(F_\omega)\,r^{n-1}+
C\,\eta\,
{\mathcal{H}}^{n-1}(F_\omega)\,r^n\le
C\,{\mathcal{H}}^{n-1}(F_\omega)\,r^{n-1} \big( 1+\eta r\big).
\end{equation}
To prove that the total energy has in fact decreased, we now check that
the cut procedure produces a considerable gain at the other
levels~$\{\omega\cdot x\in[ C_o+n,C_o+W_r]\}$.
For this, notice that
\begin{equation}\label{TOAK}
{\mbox{$\{\omega\cdot x\in[ C_o+n,C_o+W_r]\}$
contains at least~$c\,W_r\,{\mathcal{H}}^{n-1}(F_\omega)$ cubes
of~${\mathcal{Q}}_{\rm{{a.Wh}}}$.}}\end{equation} In each of these cubes,
the cut has produced an energy gain, due to the $r$-perimeter,
and possibly an energy loss due to the bulk energy of~$g$.
{F}rom~\eqref{PER:cube}, we know that the energy gain in each of these cubes
is at least~$c r^{n-1}$, up to renaming~$c>0$. On the other hand,
from~\eqref{g int men} and the fact that the cube belongs to~${\mathcal{Q}}_{\rm{{a.Wh}}}$,
we deduce an upper bound of the bulk energy loss in each cube of the form~$C\eta r^n$,
for some~$C>0$.
Hence, the variation of energy in each of these cubes is of the form~$-c r^{n-1}+C\eta r^n$
(which is negative for small~$\eta$).

Summarizing, and recalling~\eqref{TOAK}, we have that
the cut procedure has produced in~$\{\omega\cdot x\in[ C_o+n,C_o+W_r]\}$
an energy variation bounded from above by
$$ c\,W_r\,{\mathcal{H}}^{n-1}(F_\omega)
(-c r^{n-1}+C\eta r^n)\le c\,W_r\,{\mathcal{H}}^{n-1}(F_\omega)\,r^{n-1}(-1+C\eta r),$$
up to renaming~$c$ and~$C$.
{F}rom this and~\eqref{jiIKJAOJAHFUYHGAKJHIS}, up to renaming constants line after line,
we obtain that
the variation of the energy produced by the cut is in total bounded from above by
\begin{eqnarray*} &&
C\,{\mathcal{H}}^{n-1}(F_\omega)\,r^{n-1} \big( 1+\eta r\big)+
c\,W_r\,{\mathcal{H}}^{n-1}(F_\omega)\,r^{n-1}(-1+C\eta r)\\
&\le&
{\mathcal{H}}^{n-1}(F_\omega)\,r^{n-1}\big( C+C\eta r-c W_r+CW_r \,\eta r\big)
\\ &\le& {\mathcal{H}}^{n-1}(F_\omega)\,r^{n-1}\big( C+CW_r\,\eta r-c W_r\big).
\end{eqnarray*}
Since, by the minimal property of~$E^*_{\omega,M}$, this
energy variation has to be positive, we conclude that
$$ 0\le C+CW_r\,\eta r-c W_r\le C+ W_r\,(C\,\eta-c)$$
and thus, for small~$\eta$, we obtain that~$W_r$ is bounded uniformly,
by a constant independent of~$r$.

This proves~\eqref{STRISCIA3}
for the families of cubes in~${\mathcal{Q}}_{\rm{{a.Wh}}}$
(the cases of the cubes in~${\mathcal{Q}}_{\rm{{a.Bl}}}$ is similar).

{F}rom~\eqref{STRISCIA3}, one can exploit the methods in~\cite{MR1852978}, namely
find
that there exists a uniform~$M_0>0$ such
that if~$M\ge M_0$, then~$
E^*_{\omega,M}=E^*_{\omega,M_0}$ and then, checking that the minimal minimizer
is stable with respect to multiples of the period, establish that it is a Class~A
minimizer, thus completing the proof of Theorem~\ref{PLANELIKE:TH} 
for rational slopes~$\omega$.
\smallskip

{\bf Step 2: planelike minimizers with  irrational slopes}. 
Since the quantity~$M$  
is a universal constant, independent of~$n$, in order  to construct
minimizers with an irrational slope~$\omega\in S^{n-1}$ we 
approximate~$\omega$ with rational frequencies~$\omega_k$, which  produce planelike minimizers~$E^*_{\omega_k}$
and then pass to the limit in~$k$, using  Theorem \ref{promin}, which applies in particular to the Class~A planelike minimizers.
\end{proof} 

%%%%%%%%%%%%%%%%%%%%%%%%%%%%%%%%%%%%%%%%%%%%%%%%%%%%%%%%%%%%%%%%%%%%%%%%%%%%%%%%%%%%%%%%%%%%%%%%%%%%%%%%%%%%%%%%%%%%

%%%%%%%%%%%%%%%%%%%%%%%%%%%%%%%%%%%%%%%%%%%%%%%%%%%%%%%%%%%%%%

\end{document}